\documentclass[a4paper,11pt,reqno]{amsart}
\usepackage{geometry} 
\usepackage{amsmath,amsthm,amssymb}                
\usepackage{graphicx, color}
\usepackage{epstopdf}
\usepackage{pdfsync}
\usepackage{mathabx}

\usepackage[T1]{fontenc}
\usepackage[utf8]{inputenc}
\usepackage{lmodern}
\usepackage[english]{babel}

\usepackage{microtype} 
\newtheorem{definition}{Definition}[section]
\newtheorem{lemma}[definition]{Lemma}
\newtheorem{theorem}[definition]{Theorem}
\newtheorem{proposition}[definition]{Proposition}
\newtheorem{corollary}[definition]{Corollary}
\newtheorem{remark}[definition]{Remark}
\newtheorem{example}[definition]{Example}

\numberwithin{equation}{section}

\def\e{\varepsilon}

\def\dx{\,dx}

\def\wto{\rightharpoonup}
\begin{document} 

    \title[Stability in the homogenization of Hamilton-Jacobi equations]{A variational approach to the stability in the homogenization of some Hamilton-Jacobi equations}

    \author[A.~Braides]{Andrea Braides}
    \address[A.~Braides]{SISSA, via Bonomea 265, Trieste, Italy}
    \curraddr{Department of Mathematics, University of Rome Tor Vergata, Rome, Italy}
    \email{abraides@sissa.it}
    
     \author[G.~Dal Maso]{Gianni Dal Maso}
    \address[G.~Dal Maso]{SISSA, via Bonomea 265, Trieste, Italy}
    \email{dalmaso@sissa.it}
  
      \author[C.~Le Bris]{Claude Le Bris}
    \address[C.~Le Bris]{\'Ecole Nationale des Ponts et Chauss\'ees and INRIA, MATHERIALS project-team,
77455~Marne-La-Vall\'ee Cedex 2, France}
    \email{claude.le-bris@enpc.fr}
  
  %  \date{\today}

    \thanks{\textsc{Acknowledgements.} The first two authors are members of GNAMPA of INdAM.
 This article is based on work supported by the National Research Project PRIN 2022J4FYNJ  ``Variational methods for stationary and evolution problems with singularities and interfaces"
 funded by the Italian Ministry of University and Research. The first two authors thank INRIA for the kind hospitality in Paris.
 The third author thanks SISSA for its hospitality. We thank Yves Achdou for his remarks on a preliminary version of the manuscript.
}
    \keywords{Homogenization, stability, perturbation, $\Gamma$-convergence, Hamilton--Jacobi equations}
    \subjclass[2020]{49J45, 35F21, 35B40, 35B20, 35B35, 35B27}

   \begin{abstract}We investigate the stability with respect to homogenization of
classes of integrals arising in the control-theoretic interpretation of some 
Hamilton--Jacobi equations. The prototypical case is
the homogenization of energies with a Lagrangian consisting of the sum of a kinetic term and
a highly oscillatory potential $V =V_{\rm per}+ W$, where $V_{\rm per}$ is periodic and $W$ is a nonnegative perturbation thereof. We assume that $W$ has zero average in tubular domains oriented along a dense set of directions. Stability then holds true; that is, the resulting homogenized functional is identical to that for $W= 0$. We consider various extensions of this case.
As a consequence of our results, we obtain stability for the homogenization of some steady-state and time-dependent, first-order Hamilton--Jacobi equations with convex Hamiltonians and perturbed periodic potentials. Finally, we show with an example that, for negative $W$, stability may not hold. Our study revisits and, depending on the different assumptions, complements results obtained by P.-L. Lions and collaborators using PDE techniques.\end{abstract}

%\begin{abstract} We investigate the stability with respect to homogenization of
%classes of integrals arising in the control-theoretic interpretation of some 
%Hamilton--Jacobi equations. The prototypical case is
%the homogenization of energies with a Lagrangian consisting of the sum of a kinetic term and
%a highly oscillatory potential $V =V_{\rm per}+ W$, where $V_{\rm per}$ is periodic and $W$ is a nonnegative perturbation thereof. We assume that $W$ has zero average in tubular domains oriented along a dense set of directions. Stability then holds true; that is, the resulting homogenized functional is identical to that for $W= 0$. We consider various extensions of this case.
%As a consequence of our results, we obtain stability for the homogenization of some steady-state and time-dependent, first-order Hamilton--Jacobi equations with convex Hamiltonians and perturbed periodic potentials. Finally, we show with an example that, for negative $W$, stability may not hold. Our study revisits and, depending on the different assumptions, complements results obtained by P.-L. Lions and collaborators using PDE techniques.
%
%{\bf MSC codes:} 49J45, 35F21, 35B40, 35B20, 35B35, 35B27
%
%{\bf Keywords:} Homogenization, stability, perturbation, $\Gamma$-convergence, Hamilton--Jacobi equations
%
%\end{abstract}
\maketitle

\section{Introduction}
The asymptotic behaviour of viscosity solutions $U_\e$ of Hamilton-Jacobi equations of the form 
$$
\begin{cases}\partial_t U_\e(x,t) + H_{\rm per}\big(\frac{x}\e, \nabla U_\e(x,t)\big)=0,\\
U_\e(x,0)= \Phi(x),
\end{cases}
$$
with $H_{\rm per}$ periodic in the first variable, has been first studied by Lions, Papanicolaou, and Varadhan \cite{LPV}, who proved that such solutions converge uniformly as $\e\to 0$ to the solution $U$ of a {\em homogenized problem} of the form 
\begin{equation}\label{U-hom}
\begin{cases}\partial_t U(x,t) + H_{\rm hom}(\nabla U(x,t))=0,\\
U(x,0)= \Phi(x).
\end{cases}
\end{equation}
Similar statements hold for steady-state Hamilton-Jacobi equations (see e.g.~\cite{MR1159184}).

In this paper we consider a {\em stability issue} for the homogenization of Hamilton-Jacobi equations, addressing the following question: {\em what hypotheses on a perturbation $W$ ensure that viscosity solutions $\widetilde U_\e$ of equations of the form 
$$
\begin{cases}\partial_t \widetilde U_\e(x,t) + H_{\rm per}\big(\frac{x}\e, \nabla \widetilde U_\e(x,t)\big)- W\big(\frac{x}\e)=0,\\
\widetilde U_\e(x,0)= \Phi(x)
\end{cases}
$$
converge to the same $U$ solution of \eqref{U-hom}? }
Some answers to this question have been given by Achdou and Le Bris \cite{MR4643677}, who show that negative perturbations may lead to instability; that is, convergence to a different limit. In unpublished works by Lions and Souganidis, some conditions on positive $W$ are given ensuring stability (see the video presentation \cite{PLL-college}).
We note that both these results treat convex Hamiltonians, while {\em periodic} homogenization using the theory of viscosity solutions does not require such an assumption.

In the case of Hamiltonians $H_{\rm per}(x,\xi)$ convex and coercive in the variable $\xi$,  the stability question for Hamilton-Jacobi equations is related to a corresponding stability question for functionals in terms of the corresponding Lagrangian $L_{\rm per}$. Indeed, it is known that periodicity guarantees the $\Gamma$-convergence of the functionals
$$
F_\e(u)= \int_0^1 L_{\rm per} \Big(\frac{u(t)}\e, u'(t)\Big)dt
$$
to a homogenized functional
$$
F_{\rm hom}(u)= \int_0^1 L_{\rm hom} (u'(t))dt,
$$
whose homogenized Lagrangian is the one corresponding to the homogenized Hamiltonian $H_{\rm hom}$. This correspondence is ensured by the fact that the viscosity solutions $U_\e$ can be written in terms to the value function defined as a minimum for $F_\e$ through the Lax--Hopf formula. As a result, the convergence of $U_\e$ can be deduced using the Fundamental Theorem of  $\Gamma$-convergence on the convergence of minima.

The stability question for Hamilton--Jacobi equations can be then formulated as a stability question with respect to $\Gamma$-convergence: {\em what hypotheses on a perturbation $W$ ensure that the $\Gamma$-limit of 
\begin{equation}\label{giep}
G_\e(u)= \int_0^1\Big( L_{\rm per} \Big(\frac{u(t)}\e, u'(t)\Big)+ W\Big(\frac{u(t)}\e\Big)\Big)dt
\end{equation}
is still the functional $F_{\rm hom}$ (that is, the one given by the periodic case when $W=0$)?} Such a $\Gamma$-convergence question can be generalized and answered for general Lagrangians also depending on $t$, but such generalizations do not have an immediate connection with the Hamiltonian viewpoint. We note that in treating solutions of Hamilton--Jacobi equations we will use particular cases of results from the PDE literature, that apply to generic Hamiltonians %$H(x,\xi)$ (and even to $H(x,y,u,\xi)$) 
and do not make use of the specific form assumed. 
%In particular some of the results that the PDE approach can accommodate treat non-convex Hamiltonians \cite{Lions}, contrary to the $\Gamma$-limit approach. 

When $W\ge 0$, the condition we find is an integral condition on $W$. In the case of bounded $W$, this can be stated as 
\begin{equation}\label{condition}
\lim_{R\to +\infty} \frac1R\int_{B_R\cap S^r_\xi}W(x)\,dx=0
\end{equation}
for all $\xi$ in a dense subset $\Xi$ of $\mathbb R^d\setminus \{0\}$ and $r>0$; that is, the average of $W$ is zero on stripes with a given direction in a dense set (Theorem \ref{main-d}). In the one-dimensional case $d=1$, the condition simplifies in
$$
\lim_{R\to +\infty} \frac1R\int_{-R}^R W(x)\,dx=0
$$
(Theorem \ref{main-1}).
We show with an example that the average condition can indeed be required to hold only for a countable set of directions and fail otherwise. Moreover, if $d>1$ we can also treat unbounded $W$ under some uniform local integrability condition. We note that the condition on $W$ is more general than those previously considered, but, as is common for $\Gamma$-convergence results, the information we obtain is weaker since we do not give a corrector result. Condition \eqref{condition} can be compared with 
 $$
\lim_{R\to +\infty} \frac1{R^d}\int_{B_R}W(x)\,dx=0
$$
considered by the authors for the stability of elliptic homogenization \cite{BDMLB}; that is, that the average of the perturbation on the whole space is $0$. Condition \eqref{condition} highlights that for Hamilton--Jacobi equations the perturbation needs to be small on one-dimensional like sets. 

\smallskip
We give a brief description of the arguments of the proof. Since $W\ge0$ the stability result for the $\Gamma$-limit reduces to the proof of an upper bound. The main observation is that it is sufficient to treat the case of piecewise-affine target functions with slopes in the dense set of directions $\Xi$, and that the construction of recovery sequences in the periodic case requires the use of a finite number of correctors. The sequences obtained using these correctors may lead to a large contribution of the additional term involving $W$, so cannot be used as recovery sequences for the perturbed energies, but, using the zero-average condition above, we may choose careful small variations of these correctors on which the contribution of $W$ is small, and use such modified correctors to construct recovery sequences. In the one-dimensional case such modifications are not possible, but we directly show that in this case the contribution of $W$ is small on the original recovery sequences. We note that  these arguments are completely different from those used for elliptic homogenization in  \cite{BDMLB}, that rely on localization techniques and higher-integrability results. 
 
\smallskip
In order to apply the result also to steady-state Hamilton--Jacobi equations, we additionally address the stability of integrals of the form
$$
\int_0^{+\infty} \Big(L_{\rm per} \Big(\frac{u(t)}\e, u'(t)\Big)+W \Big(\frac{u(t)}\e\Big)\Big)e^{-\lambda t}\,dt.
$$
Since results for such energies are not common in the literature, we prove a general $\Gamma$-convergence theorem relating $\Gamma$-convergence on finite intervals and on the half-line (Section \ref{stab-hl}).
The applications to the stability of Hamilton--Jacobi equations when $W$ is non-negative and satisfies \eqref{condition} are finally obtained as a product of the previous results in Section \ref{stab-HJ}, both in the steady-state and evolutionary cases.

In the stability results we use non-negative perturbations $W$. We note that the sign condition on $W$ cannot be dropped altogether. In the simplest case, when $L_{\rm per}(x,\xi)=L(\xi)=L_{\rm hom}(\xi)$ is independent of $x$ and $W\le 0$ and tends to $0$ at infinity, we show that the $\Gamma$-limit of $G_\e$ defined in \eqref{giep} is given by
$$
G_{\rm hom}(u)= \int_0^1 L_{\rm hom} (u'(t))dt+\inf W|\{t: u(t)=0\}|,
$$
which is strictly lower than $F_{\rm hom}(u)$ if $|\{t: u(t)=0\}|>0$ (Section \ref{nega}).

\medskip
For the sake of clarity in the presentation of the results and their proofs, we will treat a particular form of the Lagrangians (and of the Hamiltonians); namely, in the notation used above,
$$
L_{\rm per }(x,\xi)= |\xi|^2+ V_{\rm per} (x).
$$
This form will only make it simpler to use Fenchel transforms, and set our problems in Hilbert spaces.
All the results we obtain can be extended to more general Lagrangians $L_{\rm per }$ with $ L_{\rm per }(x,\cdot)$ convex and such that there exists $r>1$ and constants $c_1,c_2>0$ such that 
$$
c_1|\xi|^r\le L_{\rm per }(x,\xi)\le c_2(1+|\xi|^r) 
$$
(see Section \ref{sec:ext}).
Indeed, the only property that we need for the Lagrangians is the existence of suitable correctors, which depends only on a polynomial growth assumption of order $r>1$ \cite{BDF}. 

%We refer to \cite{MR1484411} for the properties of the steady-state Hamilton-Jacobi equation and the corresponding representation formula using $exp (-\lambda t)$ and an integral in time for t=0 to t=+$\infty$, A good reference, where you can find all the expressions and the details of their proofs is  Chapter 3, Section 2, page 99ff of \cite{MR1484411}. In the HJB language, this is called the "infinite horizon" problem. 

%\section{Discussion of February 7, 2024 at INRIA}
\bigskip\noindent{\bf Notation.}
We use standard notation for Sobolev spaces, in particular $H^1_0$ denotes the closure of $C^\infty_c$ in $H^1$ (and $W^{1,p}_0$ its closure in $W^{1,p}$, in some remarks). We use the notation $\mathcal H^{d-1}$ for the $(d-1)$-dimensional Hausdorff (surface) measure in $\mathbb R^d$.

For the notation of $\Gamma$-convergence we refer to \cite{DM,GCB}.
Due to the form of the energies we consider, we tacitly compute $\Gamma$-limits with respect to the weak topology of $H^1$, or equivalently with respect to the strong topology of $L^2$, unless otherwise stated. We say that a sequence $\Gamma$-converges preserving the boundary or initial conditions, respectively, if it $\Gamma$-converges and for every $u$ there exists a recovery sequence with the same boundary or initial values as $u$.

\section{Stability results in the one-dimensional case}
We separately treat the case when the function $u$ is scalar. In this case the conditions on $W$ are simpler, and the proof is easier by the order structure of $\mathbb R$. 

\bigskip
We begin by defining the {\em unperturbed energies} $F_\e$.
 Let $V_{\rm per} \colon\mathbb R\to\mathbb R$ be a continuous $1$-periodic function, and for $\e\in (0,1)$ define
$$
F_\e(u)=\int_0^1 \Big(|u'(t)|^2+V_{\rm per} \Big(\frac{u(t)}\e\Big)\Big)\,dt
$$
for $u\in H^1(0,1)$. The limit as $\e\to 0$ of such functionals is described in the following theorem.

\begin{theorem}
[Homogenization Theorem (\cite{BDF}, Proposition 15.9)]\label{per-thm}
The $\Gamma$-limit of $F_\e$ is the functional $F_{\rm hom}$ defined by
\begin{equation}
F_{\rm hom}(u)=\int_0^1 f_{\rm hom} (u'(t))\,dt
\end{equation}
for $u\in H^1(0,1)$,
where $f_{\rm hom}$ is the convex function characterized by 
$f_{\rm hom}(0)=\min V_{\rm per} $ and 
\begin{equation}\label{fomxi}
f_{\rm hom}(\xi)=\min\bigg\{ |\xi|\int_0^{1/|\xi|} \big( |v'(t)+\xi|^2+ V_{\rm per} (v(t)+\xi t)\big)dt : v\in H^1_0(0,1/|\xi|) \bigg\}
\end{equation}
if $\xi\neq0$.
\end{theorem}

\begin{remark}\label{continuity}\rm
Note that $f_{\rm hom}$ satisfies the condition $|\xi|^2+\min V_{\rm per} \le f_{\rm hom}(\xi)\le |\xi|^2+\max V_{\rm per}$. By the convexity of $f_{\rm hom}$, this implies that $F_{\rm hom}$ is continuous in $H^1(0,1)$.
\end{remark} 

\begin{remark}[Periodic correctors]\rm Let $p_\xi\colon\mathbb R\to\mathbb R$ denote the $1/|\xi|$-periodic extension of a minimizer of \eqref{fomxi}, and let $w_\xi(t)= p_\xi(t)+\xi t$. Note that  $V_{\rm per}(w_\xi(t))$ is $1/|\xi|$-periodic since $V_{\rm per}(w_\xi(t+(1/|\xi|)))=V_{\rm per}(w_\xi(t)+{\rm sgn}\,\xi)= V_{\rm per}(w_\xi(t))$, and in the last equality we have used the fact that $V_{\rm per}$ is $1$-periodic.
The scaled functions $w_{\xi,\e}(t):= \e w_\xi(t/\e)= \e p_\xi(t/\e)+\xi t $ tend to $\xi t$ in $L^\infty(0,1)$ and also weakly in $H^1(0,1)$, while,  by the periodicity  and a change of variable in the integral, the functions $$t\mapsto |w_{\xi,\e}'(t)|^2+ V_{\rm per} \big(\frac{w_{\xi,\e}(t)}\e\big)= |p'_{\xi}(\tfrac{t}\e)+\xi|^2+  V_{\rm per} \big(w_{\xi}(\tfrac{t}\e)\big),$$ weakly$^*$ converge to the average $\int_0^1\big( |p'_{\xi}(t)+\xi|^2+  V_{\rm per} \big(p_{\xi}(t)+\xi t\big)\big)dt=
f_{\rm hom}(\xi)$ in $L^\infty(0,1)$. 
\end{remark}

\smallskip\goodbreak
The {\em perturbed energies} $G_\e$ will be defined as follows.
Given $W\colon\mathbb R\to\mathbb [0,+\infty)$ a Borel function we define
$$
G_\e(u)=\int_0^1 \Big(|u'(t)|^2+V_{\rm per} \Big(\frac{u(t)}\e\Big)+W\Big(\frac{u(t)}\e\Big)\Big)\,dt
$$
for $u\in H^1(0,1)$, which is well defined because any such $u$ are continuous. We can now state and prove the main result of this section.

\begin{theorem}[Stability Theorem]\label{main-1}
Let $W\colon\mathbb R\to\mathbb R$ be a Borel function such that 
\begin{equation}\label{2}
W\ge 0\quad \hbox{ and }\quad
\lim_{R\to +\infty} \frac1R\int_{-R}^R W (s)\,ds=0;
\end{equation}
then
\begin{equation}
\Gamma\hbox{-}\lim_{\e\to 0} G_\e =\Gamma\hbox{-}\lim_{\e\to 0} F_\e.
\end{equation}
\end{theorem}

\begin{proof} Let $G'':=\Gamma\hbox{-}\limsup_{\e\to 0} G_\e$, and let $F_{\rm hom}$ be given by Theorem \ref{per-thm}.  Since $W \ge 0$ it suffices to prove that $G''\le F_{\rm hom}$.
We start by proving this inequality at the function $u(t)=\xi t$, with $\xi\neq 0$.  Thanks to the 
continuity of $F_1$ with respect to the strong convergence in $H^1(0,1)$, we can choose
a  piecewise-affine $\frac1{|\xi|}$-periodic function ${p^\delta_\xi}$ that minimizes the problem in \eqref{fomxi} up to a small error $\delta>0$; that is, ${p^\delta_\xi}(0)={p^\delta_\xi}(1/|\xi|)=0$ and
\begin{equation}\label{pexi}
 |\xi|\int_0^{1/|\xi|} \big( |(p^\delta_\xi)'(t)+\xi|^2+ V_{\rm per} ({p^\delta_\xi}(t)+\xi t)\big)dt\le f_{\rm hom}(\xi)+\delta.
\end{equation}
We additionally may assume that $(p^\delta_\xi)'+\xi \neq 0$ almost everywhere since piecewise-affine functions satisfying this condition are strongly dense in $H^1(0,1)$. We set $u_{\e,\delta}(t) =\e {p^\delta_\xi}(t/\e) +\xi t$. We can then estimate 
\begin{eqnarray*}
\limsup_{\e\to 0} \int_0^1 W \Big(\frac{u_{\e,\delta}(t)}\e\Big)dt
&=&\limsup_{\e\to 0} \int_0^1 W \Big({p^\delta_\xi}\Big(\frac{t}\e\Big) +\xi \frac{t}\e\Big)dt\\
&=&\limsup_{\e\to 0} \e \int_0^{1/\e} W \big({p^\delta_\xi}(s) +\xi s\big)ds\\
&=&\limsup_{R\to +\infty} \frac1R \int_0^{R} W \big({p^\delta_\xi}(s) +\xi s\big)ds.
\end{eqnarray*}
If $(a,b)$ is an interval where ${p^\delta_\xi}$ is affine, by the change of variable $x= {p^\delta_\xi}(s) +\xi s$ we obtain
$$
\int_a^b  W\big({p^\delta_\xi}(s) +\xi s\big)ds= \frac1{(p^\delta_\xi)'+\xi}\int_{{p^\delta_\xi}(a) +\xi a}^{{p^\delta_\xi}(b) +\xi b} W(x)\dx.
$$
Note that if $s\mapsto {p^\delta_\xi}(s) +\xi s$ is monotone, then we can estimate
$$
\int_0^{\frac{n}{|\xi|}}W\big({p^\delta_\xi}(s) +\xi s\big)ds\le \max\Big\{\frac1{|(p^\delta_\xi)'+\xi|}\Big\}
\int_{-n}^n W(x)\dx.
$$
By the periodicity of ${p^\delta_\xi}$ we then obtain that
\begin{equation}\label{wexi}
\limsup_{R\to +\infty} \frac1R \int_0^{R} W\big({p^\delta_\xi}(s) +\xi s\big)ds
\le C\limsup_{R\to +\infty} \frac1R \int_{-R}^{R} W(x)\dx=0,
\end{equation}
with $C=C(\xi,\delta)=\frac{1}{|\xi|}\max\big\{\frac1{|(p^\delta_\xi)'+\xi|}\big\}$.
In the general case, this inequality holds with $C$ replaced by $CN$, 
where $N$ is the number of changes of sign of the derivative of $s\mapsto {p^\delta_\xi}(s) +\xi s$ in a period. 

Since $u_{\e,\delta}$ tends to $u(t)=\xi t$ weakly in $H^1(0,1)$ as $\e\to 0$ since the average  of $(p^\delta_\xi)'$ vanishes by periodicity, and $t\mapsto |u_{\e,\delta}'(t)|^2+ V_{\rm per} \big(\frac{u_{\e,\delta}(t)}\e\big)$ weakly$^*$ converges to the constant 
$$
 |\xi|\int_0^{1/|\xi|} \big( |(p^\delta_\xi)'(t)+\xi|^2+ V_{\rm per} ({p^\delta_\xi}(t)+\xi t)\big)dt 
 $$
in $L^\infty(0,1)$ by periodicity of $p^\delta_\xi$ and a change of variable in the integral, 
first by  \eqref{wexi} we  have
$$
\limsup_{\e\to 0} G_\e(u_{\e,\delta})=\limsup_{\e\to 0} F_\e(u_{\e,\delta}).
$$
Next, successively using \eqref{pexi} and \eqref{fomxi}, we bound the right-hand side from above
$$
\limsup_{\e\to 0} F_\e(u_{\e,\delta})\le f_{\rm hom}(\xi)+ \delta=F_{\rm hom}(u)+\delta,
$$
while, given that $u_{\e,\delta}$ tends to $u$ as $\e\to 0$ the left-hand side is bounded from below as
$$
G''(u)\le \limsup_{\e\to 0} G_\e(u_{\e,\delta}).
$$
Finally, letting $\delta\to 0$ this yields the desired inequality $G''\le F_{\rm hom}$ for $u$.

To deal with the case $u(t)= \xi t+q$ with $\xi\neq 0$ and $q\in\mathbb R$, we slightly modify the previous construction. Indeed, with fixed $\e>0$, we let $t^\e_0=\min\{t\in [0,1]: u(t)\in\e\mathbb Z\}$
and $t^1_\e= t^0_\e +k_\e\frac\e{|\xi|}$, where $k_\e$ is the largest integer such that $t^0_\e +k_\e\frac\e{|\xi|}\le 1$. We then define
$$
u_{\e,\delta}(t)=\begin{cases} \xi t +q & \hbox{ if } 0\le t\le t^0_\e
\\
\e {p^\delta_\xi}(\frac{t-t^0_\e}\e) +\xi t +q & \hbox{ if } t^0_\e\le t\le t^1_\e
\\
\xi t +q & \hbox{ if }  t^1_\e\le t\le 1.
\end{cases}
$$
Since $t^0_\e\to 0$ and $t^1_\e\to 1$ as $\e\to 0$ the same computation as above proves that $G''(u)\le F_{\rm hom} (u)$.

Noting that in the previous computation the recovery sequence attains the same values as $u$ at the endpoints of the interval $[0,1]$, we can exhibit a recovery sequence for each piecewise-affine target function $u$ such that $u'\neq 0$ almost everywhere by repeating the construction above in each interval where $u$ is affine. This leads to the inequality $G''(u)\le F_{\rm hom} (u)$ for each such functions.

Finally, by Remark \ref{continuity}, the density of piecewise-affine functions $u$ such that $u'\neq 0$ almost everywhere, and the lower-semicontinuity of $G''$, we obtain  $G''(u)\le F_{\rm hom} (u)$ for every function $u\in H^1(0,1)$.
\end{proof}

\section{Stability results in the higher-dimensional case}\label{High-case-sect}

Let $d>1$, let $V_{\rm per} \colon\mathbb R^d\to \mathbb R$ be a continuous $1$-periodic function, and for $\e\in (0,1)$  define the {\em unperturbed energies} 
$$
F_\e(u)=\int_0^1 \Big(|u'(t)|^2+V_{\rm per} \Big(\frac{u(t)}\e\Big)\Big)\,dt
$$
for $u\in H^1((0,1);\mathbb R^d)$.

The following result is proven in \cite[Theorem 15.3]{BDF}. 

\begin{theorem}[Homogenization Theorem]\label{hom-th}
The $\Gamma$-limit of $F_\e$ is the functional $F_{\rm hom}$ defined by
\begin{equation}
F_{\rm hom}(u)=\int_0^1 f_{\rm hom} (u'(t))\,dt
\end{equation}
for $u\in  H^1((0,1);\mathbb R^d)$,
where 
\begin{equation}\label{fomxi-d}
f_{\rm hom}(\xi)=\lim_{T\to+\infty}\frac1T\min\bigg\{\int_0^{T} \big( |v'(t)+\xi|^2+ V_{\rm per} (v(t)+ t\xi )\big)dt : v\in H^1_0((0,T);\mathbb R^d) \bigg\}.
\end{equation}
\end{theorem}
In particular, from \eqref{fomxi-d} it follows that $f_{\rm hom}(0)=\min V_{\rm per}$. 
Note that contrary to the scalar case, we cannot reduce to a periodic cell problem since the functions $ t\mapsto V(x+t\xi)$ are quasiperiodic but not periodic.

\begin{remark}[Almost-periodic piecewise-affine almost-correctors]\label{ap-corr}\rm Given $\xi\in\mathbb R^d$, formula \eqref{fomxi-d} and the periodicity of $V_{\rm per}$ ensure the existence of almost-correctors $p^\delta_\xi$, in a sense that will be made precise below. With fixed $\delta>0$ there exists $\eta=\eta_\delta>0$ such that $|V_{\rm per}(x+y)-V_{\rm per}(x)|<\delta$ for all $x\in\mathbb R^d$ and $|y|<\eta$. By the periodicity of $V_{\rm per}$ we then have that if $\tau>0$ is such that there exists $z\in\mathbb Z^d$ with $|\tau\xi-z|<\eta_\delta$ then 
\begin{equation}\label{ergo}
|V_{\rm per} (x+\tau\xi)-V_{\rm per} (x)|\le \delta \hbox{ for all }x\in\mathbb R^d.
\end{equation}
By well-known facts of ergodic theory on the torus, there exists $L_\delta>0$ such that every interval of length $L_\delta$ contains a $\tau$ satisfying \eqref{ergo}.

We fix 
\begin{equation}\label{tedd}
T\ge \frac{L_\delta+1}\delta
\end{equation} and a piecewise-affine function $p^\delta_\xi\in H^1_0((0,T);\mathbb R^d)$
such that
$$
\frac1T\int_0^{T} \big( |(p^\delta_\xi)'(t)+\xi|^2+ V_{\rm per} (p^\delta_\xi(t)+\xi t)\big)dt \le f_{\rm hom}(\xi)+\delta.
$$

By the ergodicity property recalled above, and since we may assume $L_\delta>1$, we can construct a sequence $T_i\in\mathbb R$ with 
\begin{equation}\label{tii}
T_0=0\quad\hbox{ and }\quad T_i+T+1\le T_{i+1}\le T_i+T+ L_\delta
\end{equation}
such that \eqref{ergo} holds for $\tau=T_i$, and extend $p^\delta_\xi$ by translation on  each $[T_i, T_i+T]$; that is $p^\delta_\xi(t)=p^\delta_\xi(t-T_i)$, and as $0$ on the remaining intervals. For use in the following proofs, we now introduce a more detailed notation for the almost-correctors. 
There exist a finite family $\xi_1,\ldots,\xi_N\in\mathbb R^d$ and a subdivision of $[0,T]$ by times $0=a_0<a_1<\ldots<a_N=T$, and such that 
$p^\delta_\xi$ is affine with gradient $\xi_j$ on $(a_{j-1},a_j)+T_i$, with $T_i$ as in \eqref{tii}. Furthermore, by continuity we may assume that we choose $p^\delta_\xi$ such that $\xi_j+\xi\neq 0$.

Note that the construction above is a particular case of the one in the proof of \cite[Theorem 15.3]{BDF} and follows from the quasi-periodicity of $t\mapsto V_{\rm per} (t\xi)$.
\end{remark}

%,and $p^\delta_\xi$ is $0$ at $0$ and on $[T_{i-1}+T,T_i]$.

As in the scalar case we define the {\em perturbed energies} $G_\e$.
Let $W\colon\mathbb R^d\to [0,+\infty)$ be a Borel function. We then set
$$
G_\e(u)=\int_0^1 \Big(|u'(t)|^2+V_{\rm per} \Big(\frac{u(t)}\e\Big)+W \Big(\frac{u(t)}\e\Big)\Big)\,dt
$$
for $u\in H^1((0,1);\mathbb R^d)$.

The hypotheses on $W$ will be more complex than in the one-dimensional case. To state them,  we introduce some notation. 
For every $x\in\mathbb R^d$ and $\rho>0$ let $B_\rho(x)$ denote the open ball with centre $x$ and radius $\rho$. If $x=0$ we omit it from the notation.
For every $\xi\in\mathbb R^d\setminus \{0\}$ and $r>0$ we define
\begin{equation}\label{esserxi}
S^r_\xi:=\{x\in\mathbb R^d: x=t\xi+z\hbox{ with } t\in\mathbb R\hbox{ and }z\in\mathbb R^d, \  |z|< r\}=\bigcup_{t\in\mathbb R}B_r(t\xi),
\end{equation}
the circular cylinder with axis in direction $\xi$ and radius $r$.

\begin{theorem}[Stability Theorem -- the higher-dimensional case]\label{main-d}
Let $W\colon\mathbb R^d\to [0,+\infty)$ be a Borel function, and assume that  
\begin{equation}\label{2-2}
\lim_{R\to +\infty} \frac1R\int_{B_R\cap S^r_\xi}W(x)\,dx=0
\end{equation}
for all $\xi$ in a dense subset $\Xi$ of $\mathbb R^d\setminus \{0\}$ and $r>0$. We also assume that
there exists $p>\frac{d}2$  such that $W\in L^p_{\rm unif}(\mathbb R^d)$; that is,
\begin{equation}\label{2-1} \sup_{y\in\mathbb R^d} \int_{B_1(y)}W^p(x)\dx<+\infty.
\end{equation}
Then
\begin{equation}
\Gamma\hbox{-}\lim_{\e\to 0} G_\e =\Gamma\hbox{-}\lim_{\e\to 0} F_\e.
\end{equation}
\end{theorem} 

The proof of the theorem will be obtained after some preliminary lemmas. Before them, we comment on the hypotheses on $W$, which we assume non-negative.
First, we note that \eqref{2-1} is satisfied if $W$ is bounded, while \eqref{2-2} is implied by the uniform convergence of $W$ to $0$ at infinity, or by
\begin{equation}\label{2+}
\lim_{R\to +\infty} \frac1R\int_{B_R}W(x)\,dx=0,
\end{equation}
which in turn is implied by $W\in L^q(\mathbb R^d)$ with $1\le q<\frac{d}{d-1}$, using H\"older's inequality.

The following example shows that \eqref{2-2} may be satisfied even if 
\begin{equation}\label{WROM}
\lim_{R\to+\infty} W(Rx)=1
\end{equation}
for almost every $x\in \mathbb R^d$, which implies, if $W$ is bounded,
$$
\langle W\rangle:=\lim_{R\to+\infty} \frac1{R^d}\int_{B_R}W(x)\,dx= |B_1|,
$$
using the Dominated Convergence Theorem.

\begin{example}\rm Let $d=2$ and let $W\colon\mathbb R^2\to [0,+\infty)$ be such that 
$W(x)=0$ if $x\in A$ and $W(x)=1$ otherwise, where $A$ is constructed as follows.

For $k\in\{2,3,\ldots\}$, we define the subsets of $[0,2\pi]$
$$
D_k:=\big\{\theta^k_h: h \hbox{ odd, } 1\le h\le 2^k\big\},\hbox{ where }  \theta^k_h=\frac{2\pi}{2^k} h,
$$
and let $D=\bigcup_{k=0}^\infty D_k$. Note that if $\theta^k_h\in D_k$ then there exists $\theta^k_{h^*}\in D_k$ such that $|\theta^k_h-\theta^k_{h^*}|=\pi$.

For each $ \theta^k_h\in D_k$ we consider a region $A^k_h$ delimited by a suitable parabola with vertex in $2^k(\cos  \theta^k_h,\sin  \theta^k_h)$ and axis given by the half line $\{ \rho (\cos  \theta^k_h,\sin  \theta^k_h): \rho\ge 2^k\}$. This parabola is constructed so that 
\begin{equation}\label{frt}
A^k_h\subset \{\rho (\cos\theta,\sin\theta): \rho\ge 2^k, |\theta-\theta^k_h|\le 4^{-k}\}.
\end{equation}
The set $A$ is defined as
$$
A=\bigcup\big\{ A^k_h: k\ge 0, h \hbox{ odd, } 1\le h\le 2^k\big\}.
$$
Given $\xi$ in the dense set $\{\rho(\cos\theta,\sin\theta): \theta\in D, \rho>0\}$, there exists $\rho>0$ and $\theta^k_h$ such that $\xi=\rho(\cos\theta^k_h,\sin\theta^k_h)$. For every $r>0$, since the regions $A^k_h$ and $A^k_{h^*}$ have the same axis as $S^r_\xi$, there exists $R_0=R_0(r,\xi)$ such that $S^r_\xi\setminus B_{R_0}\subset A^k_h\cup  A^k_{h^*}$. Hence, 
$$
 \frac1R\int_{B_R\cap S^r_\xi}W(x)\,dx\le   \frac1R |B_{R_0}|, 
$$
and condition \eqref{2-2} is satisfied. Note that also \eqref{2-1} is satisfied.

Let now 
$$
\widehat D:=\bigcap_{m=2}^\infty \bigcup_{k=m}^\infty(D_k+[-4^{-k},4^{-k}])
$$
Since $|D_k+[-4^{-k},4^{-k}]|\le 2^{-k}$ we deduce that $|\widehat D|=0$, and hence that
the set $N:=\{\rho(\cos\theta,\sin\theta): \rho\ge0, \theta\in \widehat D\}$ is negligible in $\mathbb R^2$.

We claim that if $x\in \mathbb R^2\setminus N$ then there exists $R_0=R_0(x)>0$ such that
$Rx\not\in A$ for every $R\ge R_0$. Indeed, writing $x=\rho(\cos\theta,\sin\theta)$ with $\rho>0$ and $\theta\in(0,2\pi]$, we have $\theta\not\in \widehat D$. Hence there exists $m\in\mathbb N$ such that $\theta\not\in D_k+[-4^{-k},4^{-k}]$ for all $k\ge m$. We now prove that if $R\rho\ge 2^m$ then for every $k\ge m$ we have
\begin{equation}\label{poi}
Rx\not\in\bigcup\big\{ A^k_h:  h \hbox{ odd, } 1\le h\le 2^k\big\}.
\end{equation}
Indeed, if $R\rho<2^k$ then \eqref{poi} is due to the inequality $|y|\ge 2^k$ for every $y\in A^k_h$ by the first condition in \eqref{frt}. If instead $2^k\le R\rho$, the condition on $\theta$ implies that for every $\theta^k_h\in D_k$ we have $|\theta-\theta^k_h|>4^{-k}$, and hence the second condition in \eqref{frt} ensures that $Rx\not\in A^k_h$, concluding the proof of \eqref{poi}.

On the other hand, since the axes of the parabolas defining $A^k_h$ are different from the straight line passing through the origin and $x$, there exists $R_0=R_0(x)\ge 2^m$ such that 
$$
Rx\not\in A^k_h \hbox{ for all }R\ge R_0 \hbox{ and for all }k\in\{2,\ldots,m-1\}.
$$
Together with \eqref{poi} this proves the claim.\hfill$\diamondsuit$
\end{example}

We now turn to the proof of Theorem \ref{main-d} with some preliminary lemmas.

\begin{lemma}\label{lemma1}
Let $W$ satisfy \eqref{2-1}. Then for all $\alpha>0$ such that $1<\alpha d <p$ and for all $r>0$ we have
\begin{equation}
\int_{S^{d-1}}\Big(\int_0^{r^{1/\alpha}}W(t^\alpha \theta)\,dt\Big)d\mathcal H^{d-1}(\theta)\le
r^\beta  C_{d,\alpha,p}\Big(\int_{B_r}W^p(x)\dx\Big)^{1/p},
\end{equation}
where  $\beta=\frac{p-\alpha d}{\alpha p}>0$ and $C_{d,\alpha,p}:=\Big(\frac{p-1}{p-\alpha d}\mathcal H^{d-1}(S^{d-1})\Big)^{1-\frac1p}$,
where $\mathcal H^{d-1}$ denotes the $(d-1)$-dimensional (surface) Hausdorff measure, and $S^{d-1}$ is the boundary of the unit ball in $\mathbb R^d$.
\end{lemma}

\begin{proof} Let $\gamma=\frac{\alpha d-1}p$ and $\frac1p+\frac1q=1$. Using H\"older's inequality we obtain
\begin{eqnarray*}
&&\int_{S^{d-1}}\Big(\int_0^{r^{1/\alpha}}W(t^\alpha \theta)\,dt\Big)d\mathcal H^{d-1}(\theta)\\
&\le&\Big(\int_{S^{d-1}}\Big(\int_0^{r ^{1/\alpha}}W^p(t^\alpha \theta) t^{\gamma p}\,dt\Big)d\mathcal H^{d-1}(\theta)\Big)^{1/p}\Big(\int_{S^{d-1}}\Big(\int_0^{r ^{1/\alpha}}t^{-\gamma q}\,dt\Big)d\mathcal H^{d-1}(\theta)\Big)^{1/q}.
\end{eqnarray*}
By the change of variable $\rho= t^\alpha$  we have
\begin{eqnarray*}
\int_0^{r ^{1/\alpha}}W^p(t^\alpha \theta) t^{\gamma p}\,dt=
\int_0^{r ^{1/\alpha}}W^p(t^\alpha \theta) t^{\alpha d-1}\,dt=\int_0^{r }W^p(\rho \theta) \rho^{ d-1}\,d\rho,
\end{eqnarray*}
while
\begin{eqnarray*}
\int_0^{r ^{1/\alpha}}t^{-\gamma q}\,dt= \int_0^{r ^{1/\alpha}}t^{-\frac{\alpha d -1}{p-1}}\,dt
= \tfrac{p-1}{p-\alpha d }\, r ^{ \frac{p-\alpha d }{\alpha(p-1)}}. 
\end{eqnarray*}
Inserting these equalities in the inequality obtained above, we prove the claim.
\end{proof}

\begin{lemma}\label{lemma5}
Let $W$ satisfy \eqref{2-1} with $p>\frac{d}2$, and let $\frac12<\alpha<\frac{p}d$. Let $x_0,y_0\in \mathbb R^d$, with $x_0\neq y_0$, and let $r=|y_0-x_0|$. Then there exists a  trajectory  $\gamma\in H^1((-r^{1/\alpha},r^{1/\alpha});\mathbb R^d)$ such that
\begin{eqnarray}\label{g-1}
&&\gamma(-r^{1/\alpha})= x_0,\quad \gamma(r^{1/\alpha})= y_0,\qquad\qquad\\
%\int_{-\widehat r^{1/\alpha}}^{\widehat r^{1/\alpha}}\big(|\gamma'(t)|^2+W(\gamma(t))\big)dt\le \tfrac{2\alpha^2}{2\alpha-1}r^{\frac{2\alpha-1}\alpha} +r^{\frac{p-\alpha d}{\alpha p}} \frac{ 2C_{d,\alpha,p}}{\mathcal H^{d-1}(\Theta)}\Big(\int_{B_{2r}(\frac{x_0+y_0}2)}W^p(x)\dx\Big)^{1/p}, 
\label{g-2}&&\int_{-r^{1/\alpha}}^{r^{1/\alpha}}W(\gamma(t))dt\le  r^{\frac{p-\alpha d}{\alpha p}} K_{d,\alpha,p}\Big(\int_{B_{2r}(\frac{x_0+y_0}2)}W^p(x)\dx\Big)^{1/p},\\
\label{g-3}
&&\int_{-r^{1/\alpha}}^{r^{1/\alpha}}|\gamma'(t)|^2dt
\le  \tfrac{2\alpha^{2}}{2\alpha-1}r^{\frac{2\alpha-1}\alpha},
\end{eqnarray} 
where $K_{d,\alpha,p}$ is a constant depending only on $d$, $\alpha$ and $p$.
\end{lemma}

\begin{proof} By a translation and rotation argument it is not restrictive to suppose that the middle point  $x_0+y_0$ of the segment between $x_0$ and $y_0$ is $0$, and $x_0=\frac{r}2 e_1$ and $y_0= -\frac{r}2e_1$.  Let $H=\{x\in\mathbb R^d:x_1=0\}$ 
be the symmetry hyperplane, let $C$ be the open ball in $H$ defined by
$$
C=H\cap B_r\big(\tfrac{r}2 e_1\big)=H\cap B_r\big(-\tfrac{r}2 e_1\big).
$$
and let $\Theta:= \Big\{ \theta\in S^{d-1}: \theta_1<-\frac12\Big\}$.
Note that $x\in H$ belongs to $C$ if and only if $\frac{x-\frac{r}2 e_1}{|x-\frac{r}2 e_1|}\in \Theta$. 

For $\theta=(\theta_1,\theta_2,\dots, \theta_d)$ we use the notation $\widehat\theta=(-\theta_1,\theta_2,\dots, \theta_d)$. By Lemma \ref{lemma1}, we have
\begin{equation}
\int_{\Theta}\Big(\int^0_{-r^{1/\alpha}}W\Big(\frac{r}2 e_1+(r^{1/\alpha}+t)^\alpha \theta\Big)\,dt\Big)d\mathcal H^{d-1}(\theta)\le
r^\beta  C_{d,\alpha,p}\Big(\int_{B_r(\frac{r}2 e_1)}W^p(x)\dx\Big)^{1/p},
\end{equation}
\begin{equation}
\int_{\Theta}\Big(\int_0^{r^{1/\alpha}}W\Big(-\frac{r}2 e_1+(r^{1/\alpha}-t)^\alpha\widehat\theta\Big)\,dt\Big)d\mathcal H^{d-1}(\theta)\le
r^\beta  C_{d,\alpha,p}\Big(\int_{B_r(-\frac{r}2 e_1)}W^p(x)\dx\Big)^{1/p}.
\end{equation}

Then, by the Mean Value Theorem, there exists $\theta\in\Theta$ such that 
\begin{eqnarray*}
&&\int^0_{-r^{1/\alpha}}W\Big(\frac{r}2 e_1+(r^{1/\alpha}+t)^\alpha \theta\Big)\,dt+\int_0^{r^{1/\alpha}}W\Big(-\frac{r}2 e_1+(r^{1/\alpha}-t)^\alpha\widehat\theta\Big)\,dt\\
&\le&
r^\beta 2 \frac{C_{d,\alpha,p}}{\mathcal H^{d-1}(\Theta)}\Big(\int_{B_{2r}(0)}W^p(x)\dx\Big)^{1/p}.
\end{eqnarray*}

If $\theta\in \Theta$ then we can define $t(\theta)= -\frac{1}{2\theta_1}\in [\frac12,1]$ so that $\frac{r}2 e_1+r \,t(\theta)\theta=-\frac{r}2 e_1+r \,t(\theta)\widehat \theta\in C$. We reparametrize the functions in the integrals above so that $\gamma$ defined by
\begin{equation}
\gamma(t)=\begin{cases} \frac{r}2 e_1+ (t+r^{1/\alpha} )^\alpha t(\theta)\theta & \hbox{ if $t\in [-r^{1/\alpha},0]$}\cr
-\frac{r}2 e_1+(r^{1/\alpha} -t)^\alpha t(\theta)\widehat\theta & \hbox{ if $t\in [0,r^{1/\alpha}]$}
\end{cases}
\end{equation}
is continuous in $0$ and satisfies $\gamma(-r^{1/\alpha})= \frac{r}2 e_1=x_0$ and $\gamma(r^{1/\alpha})= -\frac{r}2 e_1=y_0$.
Since $\alpha>\frac12$ we have $\gamma\in H^1((-r^{1/\alpha},r^{1/\alpha});\mathbb R^d)$.

By the estimate above and a linear change of variables, we get 
$$
\int_{-r^{1/\alpha}}^{r^{1/\alpha}}W(\gamma(t))dt\le K_{d,\alpha,p}r^\frac{p-\alpha d}{\alpha p} \Big(\int_{B_{2r}(0)}W^p(x)\dx\Big)^{1/p},
$$
with $K_{d,\alpha,p}= 2^{1+\frac1\alpha} \frac{C_{d,\alpha,p}}{\mathcal H^{d-1}(\Theta)}$. 
Finally,
$$
\int_{-r^{1/\alpha}}^{r^{1/\alpha}}|\gamma'(t)|^2dt= 2t(\theta)^2\int_0^{r^{1/\alpha}}\alpha^2 t^{2(\alpha-1)}dt\le \frac{2\alpha^2}{2\alpha-1}r^{(2\alpha-1)/\alpha},
$$
so that the claim follows.
\end{proof}

%{\bf Notes for the proof}

%\smallskip
%1) For $u=s=constant$ note that a recovery sequence can be chosen with $u_\e=s_\e$,
%and in addition such that $W(s_\e)\to0$;
%
%2) For $u=\xi t$ for $\xi\neq 0$ use recovery sequences of the form $u_\e=\xi t+ \e\, p_\xi(t/\e)$ with $p_\xi$ periodic. Note that 
%$$
%\int_0^1 W(u_\e)\,dt 
%$$
%can be bounded by a constant (depending on $p_\xi$ times the quantity in \eqref{2}).
%Possibly, we have to use a change of variables with multiplicity in this argument.
\begin{proof}[Proof of Theorem {\rm\ref{main-d}}]
We preliminarily note that many steps of the construction in the following proof simplify if $d=1$, and the proof reduces to that of the previous section with the role of correctors played by almost-correctors.

Since $W\ge 0$ we only have to prove that 
\begin{equation}
\Gamma\hbox{-}\limsup_{\e\to 0} G_\e \le \Gamma\hbox{-}\lim_{\e\to 0} F_\e.
\end{equation}
We first prove this inequality for $u(t)=t\xi  +q$ for $\xi$ belonging to the dense set $\Xi$ where \eqref{2-2} holds. To this end, for every $\lambda>0$ we will construct a sequence $u_\e\to u$ such that 
\begin{equation}\label{gls}
\limsup_{\e\to 0} G_\e(u_\e) \le (1+\lambda)\Gamma\hbox{-}\lim_{\e\to 0} F_\e(u) +\lambda.
\end{equation}
Moreover, we will take care of maintaining the boundary data; that is, $u_\e(0)=u(0)$ and $u_\e(1)= u(1)$. This will enable us to adapt this construction to the case of a piecewise-affine target function $u$.

%We first consider the case $u$ constant. Let $x_0$ be a minimizer of $V_{\rm per} $.
%By \eqref{2-} in particular we have that
%$$
%\lim_{R\to+\infty} \frac1R\int_{B_R\cap (x_0+B_\eta+\mathbb Z^d)} W(x)\dx=0
%$$
%We first consider the case $u(t)=t\xi$ {\color{red} inserire direttamente $q$?}, for which 
We first assume $q=0$ and we construct the sequence $u_\e$ from the almost-correctors $p^\delta_\xi$, and the corresponding subdivisions depending on the $T_i$ and $a_j$, introduced in Remark \ref{ap-corr}.  
We define 
$$w^\delta_\xi(t):=p^\delta_\xi(t) +t\xi,
$$
and note that
\begin{eqnarray*}
&&\hskip-1cm\int_{T_i}^{T_i+T}\Big(|(w^\delta_\xi)'(t)|^2+ V_{\rm per}(w^\delta_\xi(t))\Big)dt\\
&&= 
\int_{T_i}^{T_i+T}\Big(|(w^\delta_\xi)'(t-T_i)|^2+V_{\rm per}(p^\delta_\xi(t-T_i) +(t-T_i)\xi+T_i\xi)\Big)dt\\
&&\le\int_{0}^{T}\Big(|(w^\delta_\xi)'(t)|^2+V_{\rm per}(w^\delta_\xi(t))\Big)dt+ T\delta
\le T\,f_{\rm hom}(\xi)+2T\delta,
\end{eqnarray*}
while
\begin{eqnarray*}
\int_{T_i+T}^{T_{i+1}}\big(|(w^\delta_\xi)'(t)|^2+ V_{\rm per}(w^\delta_\xi(t))\big)dt=\int_{T_i+T}^{T_{i+1}}\big(|\xi|^2+ V_{\rm per}(t\xi)\big)dt
\le L_\delta\big(|\xi|^2+ \max V_{\rm per}\big)
\end{eqnarray*}

Hence, if we set
$u^\delta_\e(t):=\e w^\delta_\xi\big(\frac{t}\e\big)$, then, using \eqref{tedd} and \eqref{tii}, we have
\begin{equation}\label{unno}
\limsup_{\e\to 0} F_\e(u^\delta_\e, I)\le (f_{\rm hom}(\xi)+ C_\xi\delta)|I|
\end{equation}
for every interval $I$ contained in $(0,1)$, where $ C_\xi= 2+|\xi|^2+\max V_{\rm per}$ and
$$
F_\e(u, I)=\int_I \Big(|u'(t)|^2+V_{\rm per} \Big(\frac{u(t)}\e\Big)\Big)\,dt.
$$

With fixed $\e>0$ we define $i_\e$ as the largest integer $i$ such that $T_{i+1}<\frac1\e$. For $j\in\{0,\ldots,N\}$ and $i\le i_\e$ we set $a_{ij}:=a_j+T_i$ and $x_{ij}:=w^\delta_\xi(a_{ij})$. For $i<i_\e$ we also set $a_{i,N+1}:=T_{i+1}$, so that the interval $[T_i, T_{i+1}]$ is the union of the non-overlapping intervals $[a_{i\,j-1}, a_{ij}]$ for $j\in\{1,\ldots,N+1\}$. Finally, we set $a_{i_\e,N+1}:=\frac1\e$. We recall that $\{\xi_1,\ldots,\xi_N\}$ are introduced in the definition of $p^\delta_\xi$, while we set $\xi_{N+1}:=0$, so that $(p^\delta_\xi)'=\xi_j$ on $(a_{i\,j-1}, a_{ij})$.

We fix $\eta>0$, small enough to be made precise in the following, and for every $i\le i_\e$ and $j\in\{1,\ldots,N+1\}$ we construct a  function on $[a_{i\,j-1},a_{ij}]$, which takes the values of $w^\delta_\xi$ at the endpoints, so as to have a function globally defined in $[0,\frac1\e]$.
Let $\Pi_{j}$ be the hyperplane through $0$ orthogonal to $\xi_j+\xi$. For each $z\in B_\eta\cap \Pi_{j}=:B^{d-1}_{\eta,j}$ we consider the segment  parameterized as $t\mapsto x_{i\,j-1}+ z+(t-a_{i\,j-1})(\xi_j+\xi)$ for $t\in[a_{i\,j-1},a_{ij}]$, and the cylinder $C_{ij}$ in $\mathbb R^d$ obtained as the union of such segments. Then there exists $z_{ij}\in B^{d-1}_{\eta,j}$ such that 
\begin{equation}\label{f-0}
\int_{a_{i\,j-1}}^{a_{ij}} W(x_{i\,j-1}+ z_{ij}+(t-a_{i\,j-1})(\xi_j+\xi))\,dt \le C_\eta\int_{C_{ij}}  W(x)\dx,
\end{equation}
for a suitable constant $C_\eta$ depending only on $\eta$.

We fix $\alpha$ with $\frac12<\alpha<\frac{p}d$. For every $i,j$, we apply Lemma \ref{lemma5} first with $x_0=x_{i\,j-1}$ and $y_0=x_{i\,j-1}+z_{ij}$ and then with $x_0=x_{ij}+z_{ij}$ and $y_0= x_{ij}$, and we obtain that  there exist   $\gamma_{ij}\in H^1((-|z_{ij}|^{1/\alpha},|z_{ij}|^{1/\alpha});\mathbb R^d)$, $\overline\gamma_{ij}\in H^1((-|z_{ij}|^{1/\alpha},|z_{ij}|^{1/\alpha});\mathbb R^d)$ such that 
\begin{eqnarray}
&&\gamma_{ij}(-|z_{ij}|^{1/\alpha})= x_{i\,j-1},\quad \gamma_{ij}(|z_{ij}|^{1/\alpha})= x_{i\,j-1}+z_{ij},\label{f-1}\\
&&\overline\gamma_{ij}(-|z_{ij}|^{1/\alpha})= x_{ij}+z_{ij},\quad \gamma(|z_{ij}|^{1/\alpha})= x_{ij},\label{f-2}\\
&&\int_{-|z_{ij}|^{1/\alpha}}^{|z_{ij}|^{1/\alpha}}
|\gamma_{ij}'(t)|^2 dt \le \tfrac{2\alpha^2}{2\alpha-1}\eta^{\frac{2\alpha-1}\alpha}
\label{f-3}\\
&&\int_{-|z_{ij}|^{1/\alpha}}^{|z_{ij}|^{1/\alpha}}|\overline\gamma_{ij}'(t)|^2dt\le \tfrac{2\alpha^2}{2\alpha-1}\eta^{\frac{2\alpha-1}\alpha}\label{f-4}
\\
&&\int_{-|z_{ij}|^{1/\alpha}}^{|z_{ij}|^{1/\alpha}}W(\gamma_{ij}(t))dt\le C\eta^{\frac{p-\alpha d}{\alpha p}} \label{f-5}\\
&&\int_{-|z_{ij}|^{1/\alpha}}^{|z_{ij}|^{1/\alpha}}W(\overline\gamma_{ij}(t))\big)dt\le C\eta^{\frac{p-\alpha d}{\alpha p}}\label{f-6}
\end{eqnarray} 
for some constant $C$, independent of $\e,i,j$ by \eqref{2-1}.

We then consider the function $\widehat w^\delta_\xi$ defined on $[a_{i\,j-1}-2|z_{ij}|^{1/\alpha}, a_{ij}+2|z_{ij}|^{1/\alpha}]$ by setting
\begin{eqnarray*}
\widehat w^\delta_\xi(t):=\begin{cases}
\gamma_{ij}(t-a_{i\,j-1}+|z_{ij}|^{1/\alpha}) &\hbox{if }t\in [a_{i\,j-1}-2|z_{ij}|^{1/\alpha}, a_{i\,j-1}]\\
x_{i\,j-1}+ z+(t-a_{i\,j-1})(\xi_j+\xi)&\hbox{if }t\in [a_{i\,j-1},  a_{ij}]\\
\overline\gamma_{ij}(t-a_{ij}-|z_{ij}|^{1/\alpha}) &\hbox{if }t\in [a_{ij},  a_{ij}+2|z_{ij}|^{1/\alpha}].
\end{cases}
\end{eqnarray*} 
%We also set $\eta_{ij}=\big(2\widehat r_1^{1/\alpha}+2\widehat r_2^{1/\alpha}\big)^\alpha$, so that
%\begin{equation}
%\eta_{ij}\le 4^\alpha \eta \hbox{ for all }i,j.
%\end{equation}
Let $\widetilde w^\delta_\xi(t)\colon [0,+\infty)\to \mathbb R^d$ be defined on $[a_{i\,j-1},  a_{ij}]$ by scaling $\widehat w^\delta_\xi$ according to the change of variables 
$$
s=(t-a_{i\,j-1})\frac{a_{ij}-a_{i\,j-1}+4|z_{ij}|^{1/\alpha}}{a_{ij}-a_{i\,j-1}}+a_{i\,j-1}-2|z_{ij}|^{1/\alpha},
$$
so that $\widetilde w^\delta_\xi(t)= \widehat w^\delta_\xi(s)$. Since the functions match at the common endpoints of the intervals of definition, we have $\widetilde w^\delta_\xi\in H^1(0,\frac1\e;\mathbb R^d)$, with $w(0)=0$ and $w(\frac1\e)=\frac1\e\xi$. 
Note that 
\begin{eqnarray}\label{eq-n1}
 \int_{a_{i\,j-1}}^{a_{ij}}|(\widetilde w^\delta_\xi)'(t)|^2dt 
\le \kappa_\eta
 \int_{a_{i\,j-1}-2|z_{ij}|^{1/\alpha}}^{a_{ij}+2|z_{ij}|^{1/\alpha}}|(\widehat w^\delta_\xi)'(t)|^2dt, \end{eqnarray}
 where
 \begin{equation*}
\kappa_\eta:=\max_{j\in\{1,\ldots,N\}}\Bigl(\frac{a_j-a_{j-1} +4\eta^{1/\alpha}}{a_j-a_{j-1}}\Bigr)^2\vee (1+4\eta^{1/\alpha})^2.
\end{equation*}
Since by \eqref{tii} the last term in this equation takes into account the case $j=N+1$,  we have
\begin{equation}
\kappa_\eta\ge \sup_{i\le i_\e}\max_{j\in\{1,\ldots,N+1\}}\Big(\frac{a_{ij}-a_{i\,j-1}+4|z_{ij}|^{1/\alpha}}{a_{ij}-a_{i\,j-1}}\Big)^2,
\end{equation}
which justifies the estimate for the change of variable. 

The right-hand side in \eqref{eq-n1} is equal to 
\begin{eqnarray*}
&&\int_{a_{i\,j-1}-2|z_{ij}|^{1/\alpha}}^{a_{i\,j-1}}|(\widehat w^\delta_\xi)'(t)|^2dt+ \int_{a_{i\,j-1}}^{a_{ij}}|(\widehat w^\delta_\xi)'(t)|^2dt+
\int_{a_{ij}}^{a_{ij}+2|z_{ij}|^{1/\alpha}}|(\widehat w^\delta_\xi)'(t)|^2dt\\
&\le &\tfrac{2\alpha^2}{2\alpha-1}\eta^{\frac{2\alpha-1}\alpha}
+
\int_{a_{i\,j-1}}^{a_{ij}}|(w^\delta_\xi)'(t)|^2dt
+\tfrac{2\alpha^2}{2\alpha-1}\eta^{\frac{2\alpha-1}\alpha},
 \end{eqnarray*}
where we have used \eqref{f-3} and \eqref{f-4}, and the equality $(\widehat w^\delta_\xi)'=\xi+\xi_j=(w^\delta_\xi)'$ in $[a_{i\,j-1},  a_{ij}]$.
Taking into account this estimate, summing up inequalities \eqref{eq-n1}  for $i\in\{0,\ldots,i_\e\}$ and $j\in\{1,\ldots, N+1\}$, and taking into account that $i_\e T\le \frac1\e$, we then obtain 
$$
 \int_0^{1/\e}|(\widetilde w^\delta_\e)'(t)|^2dt\le \kappa_\eta\Big( \int_0^{1/\e}|(w^\delta_\e)'(t)|^2dt
 +2(N+1)\big(\tfrac{1}{\e T}+1\big)\tfrac{2\alpha^2}{2\alpha-1}\eta^{\frac{2\alpha-1}\alpha}\Big).
$$

 We set $\widetilde u^\delta_\e(t)=\e\,\widetilde w^\delta_\xi(t/\e)$,  recall that $u^\delta_\e(t)=\e w^\delta_\xi\big(\frac{t}\e\big)$. We then have
\begin{eqnarray*}
 \int_0^1|(\widetilde u^\delta_\e)'(t)|^2dt&=&
\e \int_0^{1/\e}|(\widetilde w^\delta_\e)'(t)|^2dt\\
&\le& \e   \kappa_\eta  \int_0^{1/\e}|(w^\delta_\e)'(t)|^2dt
 +2\kappa_\eta(N+1)\big(\tfrac{1}{T}+\e\big)\tfrac{2\alpha^2}{2\alpha-1}\eta^{\frac{2\alpha-1}\alpha} \\
 &\le &
  \kappa_\eta \int_0^1|(u^\delta_\e)'(t)|^2dt + C_\delta\eta^{\frac{2\alpha-1}\alpha} ,
\end{eqnarray*}
where we have taken into account that we may suppose $\e\le 1$ and $\eta$ small enough so that $\kappa_\eta\le 2$,
with $C_\delta$ a positive constant depending on $\delta$ but independent of $\e$ and $\eta$.
Noting that  $\|\widetilde w^\delta_\xi-w^\delta_\xi\|_\infty\le 2\eta^{1/\alpha}+2\eta$
we also obtain
\begin{equation*}
\int_0^1 V_{\rm per}\Big(\frac{\widetilde u^\delta_\e(t)}\e\Big)dt
\le
\int_0^1 V_{\rm per}\Big(\frac{u^\delta_\e(t)}\e\Big)dt +\omega(2\eta^{1/\alpha}+2\eta)
\end{equation*}
where $\omega$ is a modulus of continuity for $V_{\rm per}$.
We then obtain
\begin{eqnarray*}
F_\e(\widetilde u^\delta_\e)\le \kappa_\eta F_\e(u^\delta_\e)+\omega(2\eta^{1/\alpha}+2\eta)+C_\delta\eta^{\frac{2\alpha-1}\alpha} ,
\end{eqnarray*}
Hence, by \eqref{unno} we obtain
\begin{equation}
\limsup_{\e\to 0} F_\e(\widetilde u^\delta_\e)\le \kappa_\eta(f_{\rm hom}(\xi)+C_\xi \delta)+\omega(2\eta^{1/\alpha}+2\eta)+C_\delta\eta^{\frac{2\alpha-1}\alpha}.
\end{equation}
 
The perturbation term is estimated as follows. We have
\begin{eqnarray}\label{spot}
\int_0^1W\Big(\frac{\widetilde u^\delta_\e(t)}\e\Big)\,dt
=\int_0^1W\Big(\widetilde w^\delta_\e\Big(\frac{t}\e\Big)\Big)\,dt=
\e \int_0^{1/\e}W(\widetilde w^\delta_\e(s))ds.
\end{eqnarray}
%If we set $i_\e=\lfloor\frac1{T\e}\rfloor+1$, noting that $\e T_{i_\e}\ge \e i_\e T\ge 1$, 
Using \eqref{f-3}--\eqref{f-6}, and \eqref{f-0}, we have
\begin{eqnarray}\label{spett}\nonumber
\int_0^{1/\e}W(\widetilde w^\delta_\e(s))ds&\le&
 \sum_{i=0}^{i_\e} \sum_{j=1}^{N+1}\int^{a_{ij}}_{a_{i\,j-1}}W(x_{i\,j-1}+ z_{ij}+(s-a_{i\,j-1})(\xi_j+\xi))ds\\ \nonumber
 &&  + 2(i_\e+1) (N+1) C\eta^{\frac{p-\alpha d}{\alpha p}}
 \\ 
  &\le&
 C_\eta\sum_{i=0}^{i_\e} \sum_{j=1}^{N+1}\int_{C_{ij}}W(x)dx+ 2\big(\tfrac1{\e T}+1\big) (N+1) C\eta^{\frac{p-\alpha d}{\alpha p}}.\end{eqnarray}
Using this estimate in \eqref{spot}, we obtain
\begin{equation}\label{fit}
\int_0^1W\Big(\frac{\widetilde u^\delta_\e(t)}\e\Big)\,dt\le
 \e C_\eta\sum_{i=0}^{i_\e} \sum_{j=1}^{N+1}\int_{C_{ij}}W(x)dx+ C_\delta\eta^{\frac{p-\alpha d}{\alpha p}}. 
\end{equation}
%\begin{eqnarray}\label{spett}\nonumber
%&& \hskip-2cm \int_0^{1/\e}W(\widetilde w^\delta_\e(s))ds\le
% \sum_{i=0}^{i_\e} \sum_{j=1}^{N}\int^{a_{ij}}_{a_{i\,j-1}}W(x_{i\,j-1}+ z_{ij}+(s-a_{i\,j-1})(\xi_j+\xi))ds\\ \nonumber
% &&+ 2\eta (i_\e+1) N\max_{j\in\{1,\ldots,N\}} \frac1{|\xi+\xi_j|}\sup W +i_\e L_\delta \sup W\\
% \nonumber &\le&
% C_\eta\sum_{i=0}^{i_\e} \sum_{j=1}^{N}\int_{C_{ij}}W(x)dx+ 
%(i_\e+1)(\eta C_\delta+L_\delta)\sup W\\ 
%  &\le&
%C_\eta\sum_{i=0}^{i_\e} \sum_{j=1}^{N}\int_{C_{ij}}W(x)dx+ (\eta C_\delta + C L_\delta)\Big(\frac1{T\e}+2\Big)\sup W,
%\end{eqnarray}

By the boundedness of $p^\delta_\xi$ there exists $M_\delta$, independent of $i$ and $j$, such that the image of $\widetilde w^\delta_\e$ restricted to $[a_{i\,j-1},a_{ij}]$ is contained in $T_i\xi+B_{M_\delta}$. Hence, 
$$
\sum_{j=1}^{N=1}\int_{C_{ij}}W(x)dx\le (N+1) \int_{T_i\xi+B_{M_\delta}}W(x)dx.
$$
Recalling that $T_i\ge i T$, we note that at most $K_\delta:=\lfloor\frac{4M_\delta}{T|\xi|}\rfloor+1$ such balls intersect, so that, in the notation \eqref{esserxi},
\begin{equation}\label{new}
\sum_{i=0}^{i_\e} \sum_{j=1}^{N+1}\int_{C_{ij}}W(x)dx\le (N+1) K_\delta
%\Big(\Big\lfloor\frac{4M_\delta}{T|\xi|}\Big\rfloor+1\Big)
\int_{B_{R_\e}\cap S^{M_\delta}_\xi}W(x)dx,
\end{equation}
where $R_\e:=T_{i_\e}|\xi|+M_\delta$. Note that
$
R_\e\le \frac1\e|\xi|+M_\delta\le \frac2\e|\xi|
$
for $\e$ small enough, so that,
%Since, recalling \eqref{tedd}, $$\e T_{i_\e}\le \e (T+L_\delta+1)i_\e\le \Big(\frac1{T}+\e\Big)(T+L_\delta+1)\le  1+\e (T+1)+\delta+\e L_\delta,$$ we  have that 
%\begin{equation}
%\e R_\e\le (1+\e (T+1)+\delta+\e L_\delta)|\xi|+M_\delta\le K_\delta,
%\end{equation}
%where $K_\delta= (1+T+\delta+L_\delta)|\xi|+M_\delta$. 
thanks to \eqref{fit} and \eqref{new} we have
\begin{eqnarray*}
\int_0^1W\Big(\frac{\widetilde u^\delta_\e(t)}\e\Big)\,dt
\le \e C_\eta (N+1) K_\delta\int_{B_{R_\e}\cap S^M_\xi}W(x)dx
+C_\delta\eta^{\frac{p-\alpha d}{\alpha p}}\\
\le
 C_\eta\frac{2|\xi|(N+1) K_\delta}{R_\e}\int_{B_{R_\e}\cap S^M_\xi}W(x)dx
+C_\delta\eta^{\frac{p-\alpha d}{\alpha p}}
\end{eqnarray*}
for $\e$ small enough, and
\begin{eqnarray*}
\limsup_{\e\to0}\int_0^1W\Big(\frac{\widetilde u^\delta_\e(t)}\e\Big)\,dt
\le C_\delta
\eta^{\frac{p-\alpha d}{\alpha p}}\\
\end{eqnarray*}
by \eqref{2-2} and \eqref{tedd}. 
%By the arbitrariness of $\eta>0$ we then have
%\begin{eqnarray*}
%\limsup_{\e\to0}\int_0^1W\Big(\frac{\widetilde u^\delta_\e}\e\Big)\,dt
%\le  C \delta.
%\end{eqnarray*}
We then obtain
\begin{eqnarray*}
\limsup_{\e\to 0}G_\e(\widetilde u^\delta_\e)\le 
\kappa_\eta(f_{\rm hom}(\xi)+C_\xi \delta)+\omega(2\eta^{1/\alpha}+2\eta) +
C_\delta\eta^{\frac{p-\alpha d}{\alpha p}},
\end{eqnarray*}
and, noting that $\kappa_\eta\to 1$ as $\eta\to 0$, letting first $\eta\to 0$ and then $\delta\to 0$, given $\lambda>0$ we obtain \eqref{gls} for a suitable choice of the parameters $\delta$ and $\eta$. By the arbitrariness of $\lambda>0$ we obtain that
$$
\Gamma\hbox{-}\limsup_{\e\to 0} G_\e(u)\le f_{\rm hom}(\xi)=\int_0^1 f_{\rm hom}(u'(t))\,dt
$$
for affine functions $u(t)=t\xi$ with $\xi\in\Xi$.\medskip

In the case $q\neq 0$ we translate the recovery sequence $u^\delta_\e$ constructed above for $F_\e$ by $\e\lfloor \frac{q}\e\rfloor$, where this vector is defined component-wise by taking the integer parts of the components, and then use the same argument as above to construct a recovery sequence $\widetilde u^\delta_\e$ for $G_\e$. Thanks to the periodicity of $V_{\rm per}$ this sequence provides a good upper bound, but it does not match the boundary conditions, since $\widetilde u^\delta_\e(0)-u(0)= \widetilde u^\delta_\e(1)-u(1)= \e\lfloor \frac{q}\e\rfloor-q$. In order to match the boundary conditions, we can proceed similarly to the case above, first applying Lemma \ref{lemma5} with $x_0=q$ and $y_0=\e \lfloor \frac{q}\e\rfloor$, and similarly to the second endpoint. Note that we have $|x_0-y_0|\le \sqrt d$. Hence, if we let, respectively, $\gamma_0, \gamma_1\colon[-d^{\frac1{2\alpha}},d^{\frac1{2\alpha}}]\to \mathbb R^d$ be the functions given by Lemma \ref{lemma5} at both endpoints, we can define the functions $\widehat u^\delta_\e\colon [-2\e d^{\frac1{2\alpha}}, 1+2\e d^{\frac1{2\alpha}}]$ as 
$$
\widehat u^\delta_\e(t)= \begin{cases} \e \gamma_0\big(\frac{t+\e d^{\frac1{2\alpha}}}\e\big) 
& \hbox{ if }t\in [-2\e d^{\frac1{2\alpha}},0]\\
u^\delta_\e(t) & \hbox{ if }t\in [0,1]\\
\e \gamma_1\big(\frac{t-1-\e d^{\frac1{2\alpha}}}\e\big) 
& \hbox{ if }t\in [1,1+2\e d^{\frac1{2\alpha}}].
\end{cases}
$$
Note that Lemma \ref{lemma5} ensures that the contribution to the energy in the two extreme intervals is of order $\e$. Finally, we define $\widetilde u^\delta_\e$ by an affine change of variables rescaling the domain to $[0,1]$. 
\medskip

We now turn our attention to the case of a piecewise-affine target function $u$.
Namely, we assume that $0=t_0<t_1<\ldots< t_K=1$ and $\xi_j,q_j\in\mathbb R^d$ exist such that
$u(t)=t\xi_jt+q_j$ on $[t_{j-1}, t_j]$ . Moreover, we assume that the vectors $\xi_j$
belong to $\Xi$.
We can use the construction just described for $t\xi_j+q_j$ in the place of $t\xi +q$ and $[t_{j-1}, t_j]$ in the place of $[0,1]$, thus obtaining that for such functions 
\begin{equation}\label{h-0}
\Gamma\hbox{-}\limsup_{\e\to 0} G_\e(u)\le\sum_{j=1}^K (t_j-t_{j-1}) f_{\rm hom}(\xi_j)=\int_0^1 f_{\rm hom}(u'(t))dt=  \Gamma\hbox{-}\lim_{\e\to 0} F_\e(u).
\end{equation}
Note that in this argument it is essential that we are able to maintain the values at all $t_j$ in the
construction of the recovery sequences in each interval $[t_{j-1}, t_j]$, as done above. 
\medskip

Finally, if $u\in H^1((0,1);\mathbb R^d)$ then we may take a sequence of piecewise-affine functions $u_j$ strongly converging to $u$ in $H^1((0,1);\mathbb R^d)$ with $u_j'(t)\in\Xi$ for almost every $t$, and such that $u_j(0)=u(0)$ and $u_j(1)=u(1)$, and recall that the 
$\Gamma\hbox{-}\limsup$ is a lower-semicontinuous functional, so that, by \eqref{h-0} we have
\begin{eqnarray*}
\Gamma\hbox{-}\limsup_{\e\to 0} G_\e(u)&\le& \liminf_{j\to+\infty} \Big(\Gamma\hbox{-}\limsup_{\e\to 0} G_\e(u_j)\Big)\\
&\le& \liminf_{j\to+\infty} \int_0^1 f_{\rm hom}(u_j'(t))dt=
\int_0^1 f_{\rm hom}(u'(t))dt,
\end{eqnarray*}
which completes the proof of the theorem.
\end{proof}

\begin{remark}[Stability for uniformly almost-periodic potentials]\rm 
The construction of the recovery sequences in the previous proof is based on the possibility of having almost-periodic correctors. This is the case also if $V_{\rm per}$ is uniformly almost periodic; that is, it is the uniform limit of (possibly incommensurate) trigonometric polynomials. We refer to \cite[Chapter 15]{BDF} for details.
 \end{remark}

\section{Extensions}
In this section we extend the previous  results by considering more general Lagrangians with possibly different growth conditions, and by examining the case of unbounded time intervals.

\subsection{Extension to general integrands}\label{sec:ext}

{The argument in the proof of the stability result relies on showing that a suitable small variation of almost-correctors gives test functions that are negligible for the perturbation potential $W$ 
and are still recovery sequences for the unperturbed part. This argument can be repeated whenever 
almost-correctors as described in Section \ref{High-case-sect} exist, and in particular if the energy density $|\xi|^2+V(x)$ is replaced by a periodic Lagrangian $L$, as in the following result.

In the following result, given $r>1$ the $\Gamma$-limits are computed with respect to the weak topology of $W^{1,r}$, or equivalently with respect to the strong topology of $L^r$.

\begin{theorem}[Stability Theorem -- general Lagrangians]\label{main-d-L}
Let $W\colon\mathbb R^d\to [0,+\infty)$ be a bounded Borel function satisfying  \eqref{2-2}. Let $L_{\rm per}\colon \mathbb R^d\times\mathbb  R^d\to [0,+\infty)$ be a Carath\'eodory function such that $x\mapsto L_{\rm per}(x,\xi)$ is $(0,1)^d$-periodic for all $\xi$ and $r>1$, $c_1,c_2>0$ exist such that
\begin{equation}
c_1|\xi|^r\le L_{\rm per}(x,\xi)\le c_2(1+|\xi|^r)
\end{equation}
for all $(x,\xi)$, and let
\begin{equation}
F_\e(u)=\int_0^1 L_{\rm per}\Big(\frac{u(t)}\e,u'(t)\Big)dt\ \hbox{ and }\  G_\e(u)=\int_0^1\Big( L_{\rm per}\Big(\frac{u(t)}\e,u'(t)\Big)+W\Big(\frac{u(t)}\e\Big)\Big)dt
\end{equation}
be defined on $W^{1,r}((0,1);\mathbb R^d)$.
Then
\begin{equation}\label{claim-uap}
\Gamma\hbox{-}\lim_{\e\to 0} G_\e =\Gamma\hbox{-}\lim_{\e\to 0} F_\e =
\int_0^1L_{\rm hom} (u')dt,
\end{equation}
where $L_{\rm hom}$ satisfies the {\em asymptotic homogenization formula}
\begin{equation}
L_{\rm hom} (\xi)=\lim_{T\to+\infty}\frac1T\min\Big\{\int_0^{T} L_{\rm per}(v(t)+ t\xi,v'(t)+\xi)dt : v\in W^{1,r}_0((0,T);\mathbb R^d) \Big\}.
\end{equation}
\end{theorem}

\begin{proof} The Lagrangian $L_{\rm per}$ satisfies the hypotheses of \cite[Theorem 15.3]{BDF}, which states the convergence of $F_\e$ and their characterization in terms of $L_{\rm hom}$. 
The construction of almost-correctors is carried on in the proof of \cite[Proposition 15.5]{BDF}, after which the proof of Theorem \ref{main-d} can be followed word for word.
\end{proof}

\begin{remark}\rm Note that for simplicity we have required that $W$ be bounded. Otherwise, \eqref{2-1} must be assumed for some $p>\frac{r-1}rd$, for which the analog of Lemma \ref{lemma5} with $1-\frac1r<\alpha<\frac{p}d$. Note however that in the applications to Hamilton-Jacobi equations $W$ will be bounded.
\end{remark}

\begin{remark}[Time-dependent Lagrangians]\rm
The extension to Lagrangians $L_{\rm per}(t,x,\xi)$ also depending (almost-)periodically on the time variable $t\in\mathbb R$ can be carried on using the results of \cite[Section 15]{BDF}. Unless suitable continuity is required in both the $t$ and $s$ variables, it is necessary to state some additional assumptions, due to the fact that we have to consider sections of the Lagrangian; that is, functions $t\mapsto L_{\rm per}(t,\xi t, \xi )$, which are in general not periodic even if $L_{\rm per}$ is periodic.
We assume a kind of {\em uniform almost-periodicity condition}; more precisely, that for all $\xi\in\mathbb R^d$ and $\eta>0$ the sets 
\begin{eqnarray}\label{uap}\nonumber
&&{\mathcal T}^\xi_\eta:=\{\tau\in \mathbb R: |L_{\rm per}(t+\tau, x+\xi \tau,  \xi )-L_{\rm per}(t, x,  \xi )|<\eta(1+|\xi|^r)\hbox{ for all }(t,x,\xi)\}\\
&&{\mathcal S}_\eta:=\{\sigma\in \mathbb R^d: |L_{\rm per}(t, x+\sigma,  \xi )-L_{\rm per}(t, x,  \xi )|<\eta(1+|\xi|^r)\hbox{ for all }(t,x,\xi)\}
\end{eqnarray}
are uniformly dense; that is, there exists $\Lambda_\eta>0$ such that  we have   dist$(\sigma,{\mathcal S}_\eta\setminus\{\sigma\})<\Lambda_\eta$ for all $\sigma\in {\mathcal  S}_\eta$ and for all $\xi\in\mathbb R^d$ we have dist$(\tau,{\mathcal T}^\xi_\eta\setminus\{\tau\})<\Lambda_\eta$ for all $\tau\in {\mathcal T}^\xi_\eta$. Note that we do not assume that $L_{\rm per}$ is continuous, but condition \eqref{uap} holds if $L_{\rm per}$ is continuous and periodic in the first two variables with a modulus of continuity controlled by $(1+|\xi|^r)$. 

If $L_{\rm per}$ is a Carath\'eodory function satisfying \eqref{uap}  and the growth condition
\begin{equation}
c_1|\xi|^r\le L_{\rm per}(t,x,\xi)\le c_2(1+|\xi|^r)
\end{equation}
for all $(t,x,\xi)$ is satisfied, and if we define
\begin{equation}
F_\e(u)=\int_0^1 L_{\rm per}\Big(\frac{t}\e,\frac{u(t)}\e,u'(t)\Big)dt\, \hbox{ and }  \,G_\e(u)=\int_0^1\Big( L_{\rm per}\Big(\frac{t}\e,\frac{u(t)}\e,u'(t)\Big)+W\Big(\frac{u(t)}\e\Big)\Big)dt
\end{equation}
on $W^{1,r}((0,1);\mathbb R^d)$,
then the stability result  in \eqref{claim-uap} still holds, with
$L_{\rm hom}$ satisfying
\begin{equation}
L_{\rm hom} (\xi)=\lim_{T\to+\infty}\frac1T\min\Big\{\int_0^{T} L_{\rm per}(t,v(t)+ t\xi,v'(t)+\xi)dt : v\in W^{1,r}_0((0,T);\mathbb R^d) \Big\}.
\end{equation}
\end{remark}
}

\begin{remark}\rm The case $r=1$ would require a different treatment, since the functionals are not equicoercive in the weak topology of $W^{1,1}$. The standard approach would be to extend their definition to the space $BV$ of functions of bounded variation. For a periodic homogenization result in such a context we refer to \cite{MR1649453}. 
\end{remark}

\subsection{Stability in $(0,+\infty)$}\label{stab-hl}
We now consider the extension of the $\Gamma$-convergence stability result to some functionals defined on functions on the half line, in view of the applications to steady-state Hamilton-Jacobi equations.

Let $r>1$ be a fixed exponent, and let $\lambda>0$ be a given parameter. We define the space
$
W^{1,r}_\lambda((0,+\infty);\mathbb R^d)$  of all functions $u\colon(0,+\infty)\to\mathbb R^d$ such that  $u\in W^{1,r}((0,t_0);\mathbb R^d)$ for all $t_0>0$ and 
$$
\int_0^{+\infty} |u'(t)|^r  e^{-\lambda t}dt<+\infty,
$$
endowed with the norm
$$
\|u\|_\lambda=\Big(|u(0)|^r +\int_0^{+\infty} |u'(t)|^r  e^{-\lambda t}dt\Big)^{\frac1r}.
$$
We observe that $W^{1,r}_\lambda((0,+\infty);\mathbb R^d)$ is a Banach space. In the case $r=2$ we write $H^1_\lambda((0,+\infty);\mathbb R^d)$ in the place of $W^{1,2}_\lambda((0,+\infty);\mathbb R^d)$. 
%$$
%(u,v)_\lambda:=\langle u(0),v(0)\rangle +\int_0^{+\infty} \langle u'(t),v'(t)\rangle  e^{-\lambda t}dt<+\infty,
%$$
%where  $\langle z,w\rangle$ is the scalar product in $\mathbb R^d$. Correspondingly we let $\|u\|_\lambda$ denote the relative Hilbert norm.

\begin{proposition}\label{const-dense}
The set of functions $u\in W^{1,r}_\lambda((0,+\infty);\mathbb R^d)$ that are equal to $0$ outside some bounded interval is dense in $W^{1,r}_\lambda((0,+\infty);\mathbb R^d)$.
\end{proposition}

\begin{proof}
Let $u\in W^{1,r}_\lambda((0,+\infty);\mathbb R^d)$ and $\varphi\in C^\infty(\mathbb R)$ such that $\varphi(t)=1$ for $t\le 0$, $\varphi(t)=0$ for $t\ge 1$ and $|\varphi'(t)|\le 2$ for all $t$.
For all $j,k\in\mathbb N$ with $k\ge j$  let $u_{j,k}$ be defined as equal to $u$ on $(0,j)$ and to $u(j)\varphi(t-k)$ on $(j,+\infty)$. Then 
\begin{eqnarray*}
\|u_{j,k}-u\|^r_\lambda&=&\int_j^{+\infty} |u'(t)+\varphi'(t-k) u(j)|^re^{-\lambda t}dt\\
&\le & C_r\Big(\int_j^{+\infty}|u'(t)|^re^{-\lambda t}dt+ |u(j)|^r\int_k^{k+1} e^{-\lambda t}dt\Big)
\end{eqnarray*}
for some positive constant $C_r$ depending only on $r$.
This shows that $\lim\limits_{j\to+\infty}\lim\limits_{k\to+\infty} \|u_{j,k}-u\|^r_\lambda=0$ and proves the claim.
\end{proof}

For all $\e> 0$ let $L_\e\colon(0,+\infty)\times \mathbb R^d\times \mathbb R^d\to [0,+\infty)$ be Borel functions. For every bounded open interval $I$ of $(0,+\infty)$ we define
$$
\mathcal L_\e(u, I)=\int_I L_\e(t,u(t), u'(t)) dt 
$$
for $u\in W^{1,r}(I;\mathbb R^d)$.
Moreover, we also define 
$$
\mathcal L^\lambda_\e(u)=\int_0^{+\infty} L_\e(t,u(t), u'(t))  e^{-\lambda t}dt 
$$
for $u\in W^{1,r}_\lambda((0,+\infty);\mathbb R^d)$.

\begin{theorem} \label{infinite}
For all $\e> 0$ let $L_\e:(0,+\infty)\times \mathbb R^d\times \mathbb R^d\to [0,+\infty)$ be Borel functions. Assume that

{\rm(i)} there exists a Borel function $L_0\colon(0,+\infty)\times \mathbb R^d\times \mathbb R^d\to [0,+\infty)$  for all bounded open intervals $I$ of $(0,+\infty)$ the functionals $\mathcal L_\e(\cdot, I)$ defined above $\Gamma$-converge, as $\e\to 0^+$, with respect to the weak topology of $W^{1,r}(I;\mathbb R^d)$ to the functional  $\mathcal L_0(\cdot, I)$ defined by
$$
\mathcal L_0(u, I)=\int_I L_0(t,u(t), u'(t)) dt 
$$
for $u\in W^{1,r}(I;\mathbb R^d)$,

{\rm(ii)} for all bounded open intervals $I$ of $(0,+\infty)$ and for all $u\in W^{1,r}(I;\mathbb R^d)$ there exists a sequence $u_\e$ such that $u_\e= u$ at the endpoints of the interval $I$ and $\mathcal L_\e(u_\e, I)$ tends to $\mathcal L_0(u, I)$; 

{\rm(iii)}  there exists a constant $C>0$ such that
\begin{equation}
L_\e(t,0,0)\le C \hbox{ for every }\e>0, \hbox{ and }t>0,
\end{equation}

{\rm(iv)}  $L_0$ is a Carath\'eodory function satisfying
\begin{equation}
L_0(t,x,\xi)\le C (1+|\xi|^r) \hbox{ for every }t>0,\  x,\xi\in\mathbb R^d.
\end{equation}
Then the functionals $\mathcal L^\lambda_\e$ defined above $\Gamma$-converge,  as $\e\to 0^+$, with respect to the weak topology of $W^{1,r}_\lambda((0,+\infty);\mathbb R^d)$ to the functional $\mathcal L^\lambda_0$ defined by
$$
\mathcal L^\lambda_0(u)=\int_0^{+\infty} L_0(t,u(t), u'(t))  e^{-\lambda t}dt 
$$
for $u\in W^{1,r}_\lambda((0,+\infty);\mathbb R^d)$.
\end{theorem}

Preliminarily to the proof of this results, we state the following lemma

\begin{lemma}\label{lemma33}
Suppose that $L_\e$ and $L_0$  be Borel functions satisfying hypotheses {\rm(i)} and {\rm(ii)} of the previous theorem, and let $\phi:[0,+\infty)\to [0,+\infty)$ be a continuous function.
Given a bounded open interval $I$ of $(0,+\infty)$ we define
$$
\mathcal L^\phi_\e(u, I)=\int_I L_\e(t,u(t), u'(t))\phi(t) dt,\quad \hbox{ and }\quad
\mathcal L^\phi_0(u, I)=\int_I L_0(t,u(t), u'(t))\phi(t) dt,
$$
for $u\in H^1(I;\mathbb R^d)$. Then the functionals $\mathcal L^\phi_\e(\cdot, I)$  $\Gamma$-converge preserving the boundary conditions with respect to the weak topology of $W^{1,r}(I;\mathbb R^d)$ to the functional  $\mathcal L^\phi_0(\cdot, I)$.
\end{lemma}

\begin{proof} Let $u_\e\to u$; then, by (i) for every subinterval $J\subset I$ we have 
$$
m_J\int_J L_0(t,u(t), u'(t))dt\le \liminf_{\e\to 0^+}\mathcal L^\phi_\e(u_\e, J),
$$
where $m_J:=\inf_J \phi$. By covering $I$ by a finite number of disjoint intervals, of sufficiently small size, and using the uniform continuity of $\phi$, we can deduce the liminf inequality.

Conversely, taking $u_\e$ as in (ii) in $J$ we have 
$$
\limsup_{\e\to 0^+}\mathcal L^\phi_\e(u_\e, J)\le M_J\int_J L_0(t,u(t), u'(t))dt,
$$
where $M_J:=\inf_J \phi$. Since $u_\e=u$ at the endpoints of $J$ we can deduce the limsup inequality for $u$ on $I$, by covering $I$ by a finite number of disjoint intervals, of sufficiently small size, and using the uniform continuity of $\phi$.
\end{proof}

\begin{proof}[Proof of Theorem {\rm\ref{infinite}}]
Let $u_\e$ be a sequence in $W^{1,r}_\lambda((0,+\infty);\mathbb R^d)$ converging weakly to $u$.
Then $u_\e$ converges to $u$ weakly in $W^{1,r}((0,t_0);\mathbb R^d)$ for all $t_0>0$. By the lemma above with $\phi(t)=e^{-\lambda t}$, we obtain
$$
\int_0^{t_0} L_0(t,u(t), u'(t)) e^{-\lambda t}dt  \le  \liminf_{\e\to 0^+}\int_0^{t_0} L_\e(t,u_\e(t), u_\e'(t)) e^{-\lambda t}dt \le   \liminf_{\e\to 0^+}\mathcal L^\lambda_\e(u_\e).
$$
Taking the limit as $t_0\to+\infty$ we obtain 
$$
 \mathcal L^\lambda_0(u)\le   \liminf_{\e\to 0^+}\mathcal L^\lambda_\e(u_\e).
$$

To prove the upper bound we reason by density. To that end, we first note that $\mathcal L^\lambda_0$ is continuous with respect to the strong topology of $W^{1,r}_\lambda((0,+\infty);\mathbb R^d)$, thanks to (iv) and a generalized version of the Dominated Convergence Theorem. Therefore, in view of Proposition \ref{const-dense} it is sufficient to consider  a target function $u\in W^{1,r}_\lambda((0,+\infty);\mathbb R^d)$ such that there exists $t_0$ such that $u(t)=0$ if $t\ge t_0$. For all $\tau>t_0$, by Lemma \ref{lemma33} there exist a sequence $u_\e$ converging weakly in $W^{1,r}((0,\tau);\mathbb R^d)$ to $u$ 
and such that $u_\e(\tau)= u(\tau)=0$, $u_\e$ converges weakly in $W^{1,r}((0,\tau);\mathbb R^d)$,  and 
$$
\lim_{\e\to 0^+} \int_0^{\tau} L_\e(t, u_\e(t), u'_\e(t))e^{-\lambda t}dt= 
\int_0^{\tau} L_0(t, u(t), u'(t))e^{-\lambda t}dt.
$$
We extend $u_\e$ by setting $u_\e(t)=0$ if $t\ge \tau$, and compute
\begin{eqnarray*}
\Gamma\hbox{-}\limsup_{\e\to 0^+} L^\lambda_\e(u)&\le &\limsup_{\e\to 0^+} \mathcal L^\lambda_\e(u_\e) \\
&\le& \int_0^{\tau} L_0(t, u(t), u'(t))e^{-\lambda t}dt
+\limsup_{\e\to 0^+}\int_{\tau}^{+\infty} L_\e(t,0,0)e^{-\lambda t}dt\\
&\le& \mathcal L^\lambda_0(u) +C\int_{\tau}^{+\infty}e^{-\lambda t}dt,
\end{eqnarray*}
where in the last inequality we have used property (iii). By the arbitrariness of $\tau\ge t_0$ we obtain the upper bound.
\end{proof}

\begin{remark}\label{remmm}\rm  
Let $V_{\rm per} \colon\mathbb R^d\to \mathbb R$ be a continuous $1$-periodic function.
We can apply Theorem \ref{infinite} to the sequence $L_\e(t,x,\xi)= |\xi|^2+V_{\rm per}(\frac x\e)$ and to $L_0(t,x,\xi)=f_{\rm hom}(\xi)$, where $ f_{\rm hom}$ is defined in \eqref{fomxi-d}, noting that hypotheses (i) and (ii) are proven to hold in Theorem \ref{hom-th}, which is proved in $(0,1)$ for simplicity of notation, but holds in any bounded interval. As a consequence, if $\lambda>0$ and 
for $u\in H^1_\lambda((0,+\infty);\mathbb R^d)$, we define 
\begin{eqnarray}\label{fae}
&&F^\lambda_\e(u):=\int_0^{+\infty} \Big(|u'(t)|^2+V_{\rm per} \Big(\frac{u(t)}\e\Big)\Big)e^{-\lambda t}\,dt\\
&&F^\lambda_{\rm hom}(u):= \int_0^{+\infty}f_{\rm hom}(u'(t))e^{-\lambda t}\,dt,
\end{eqnarray}
we obtain that $F^\lambda_\e$ $\Gamma$-converge to $F^\lambda_{\rm hom}$ in the weak topology of 
$H^1_\lambda((0,+\infty);\mathbb R^d)$.
\end{remark}

\begin{theorem}\label{main-infinity} 
Let $W\colon\mathbb R^d\to [0,+\infty)$ be a Borel function satisfying the hypotheses of Theorem
{\rm\ref{main-d}}, and let $V_{\rm per}$ be as above. Let $F^\lambda_\e$ be defined by \eqref{fae} and $G^\lambda_\e$ be defined as 
\begin{equation}
G^\lambda_\e(u):=\int_0^{+\infty} \Big(|u'(t)|^2+V_{\rm per} \Big(\frac{u(t)}\e\Big)+W \Big(\frac{u(t)}\e\Big)\Big)e^{-\lambda t}\,dt
\end{equation}
for $u\in H^1_\lambda((0,+\infty);\mathbb R^d)$. Then 
\begin{equation}
\Gamma\hbox{-}\lim_{\e\to 0} G^\lambda_\e =\Gamma\hbox{-}\lim_{\e\to 0} F^\lambda_\e,
\end{equation}
preserving the initial conditions, with respect to the weak topology of $H^1_\lambda((0,+\infty);\mathbb R^d)$. 
\end{theorem} 

\begin{proof}
We can apply Theorem \ref{infinite},  with 
$L_\e(t,x,\xi)= |\xi|^2+V_{\rm per}(\frac x\e)+W(\frac x\e)$ and to $L_0(t,x,\xi)=f_{\rm hom}(\xi)$,
noting that (i) and (ii) hold by Theorem \ref{main-d}.
The conclusion then follows by Remark \ref{remmm}.
\end{proof}

\begin{corollary}\label{cor-con-pu} Under the hypotheses of the previous theorem, we have
$$
\lim_{\e\to 0^+} \inf_{u(0)=x} G^\lambda_\e(u)= \lim_{\e\to 0^+} \inf_{u(0)=x} F^\lambda_\e(u)= \min_{u(0)=x}F^\lambda_{\rm hom}(u).
$$
\end{corollary}

\begin{proof} For all $u\in H^1_\lambda((0,+\infty);\mathbb R^d)$ with $u(0)=x$ by Theorem \ref{main-infinity} there exists a sequence $u_\e\to u$ weakly in $H^1_\lambda((0,+\infty);\mathbb R^d)$ such that $u_\e(0)=x$ and $G^\lambda_\e(u_\e)\to F^\lambda_{\rm hom}(u)$. This implies that 
$$
\limsup_{\e\to 0^+}  \inf_{v(0)=x} G^\lambda_\e(v)\le \limsup_{\e\to 0^+} G^\lambda_\e(u_\e)= F^\lambda_{\rm hom}(u).
$$
Taking the infimum with respect to such $u$ we obtain 
$$
\limsup_{\e\to 0^+}  \inf_{v(0)=x} G^\lambda_\e(v)\le  \inf_{u(0)=x} F^\lambda_{\rm hom}(u)<+\infty.
$$
Conversely, we consider a sequence $\e_k\to 0$ such that
$$
\lim_{k\to+\infty}  \inf_{u(0)=x} F^\lambda_{\e_k}(u)= \liminf_{\e\to 0^+}  \inf_{u(0)=x} F^\lambda_{\e}(u)\le \limsup_{\e\to 0^+}  \inf_{u(0)=x} G^\lambda_\e(u)<+\infty,
$$
and correspondingly a sequence $u_k$ with $u_k(0)=x$ such that 
$$
\lim_{k\to+\infty}  F^\lambda_{\e_k}(u_k)= \liminf_{\e\to 0^+}  \inf_{u(0)=x} F^\lambda_{\e}(u).
$$
Since we have $ F^\lambda_{\e_k}(u_k)\ge \|u_k\|_\lambda^2-|x|^2$ the sequence $u_k$ is bounded in 
$H^1_\lambda((0,+\infty);\mathbb R^d)$; hence, up to subsequences $u_k$ converges to a function $u$
with $u(0)=x$ weakly in $H^1_\lambda((0,+\infty);\mathbb R^d)$. By the liminf inequality
\begin{eqnarray*}
\inf_{v(0)=x} F^\lambda_{\rm hom}(v)&\le& F^\lambda_{\rm hom}(u)\le \lim_{k\to+\infty}  F^\lambda_{\e_k}(u_k)= \liminf_{\e\to 0^+}  \inf_{v(0)=x} F^\lambda_{\e}(v)\\
&\le &\limsup_{\e\to 0^+}  \inf_{v(0)=x} F^\lambda_{\e}(v)\le \limsup_{\e\to 0^+}  \inf_{v(0)=x} G^\lambda_\e(v)\le  \inf_{v(0)=x} F^\lambda_{\rm hom}(v).
\end{eqnarray*}
This proves that $u$ is a minimizer of $F^\lambda_{\rm hom}$ with the initial condition $u(0)=x$,
and the convergence of $ \inf\limits_{v(0)=x} F^\lambda_{\e}$. The same argument proves the  convergence of $ \inf\limits_{v(0)=x} G^\lambda_{\e}$.
\end{proof}

\section{Stability for Hamilton--Jacobi equations}\label{stab-HJ}
In this section we use the $\Gamma$-convergence approach to recover some stability results  obtained using PDE techniques and presented by P.-L.~Lions in his lectures \cite{PLL-college}. Although our results require specific hypotheses on the periodic Hamiltonian, our assumptions on the non-negative perturbation $W$ are much weaker (see \eqref{2-2}) than those considered in  \cite{PLL-college}.

Using the notation of Section \ref{High-case-sect} we consider continuous and periodic $V_{\rm per}$.
In this section we suppose that the perturbation $W$ satisfies \eqref{2-2} and
\begin{equation}\label{WBUC}
\hbox{ $W$ is bounded and uniformly continuous on $\mathbb R^d$,}
\end{equation}
and use the notation 
%\begin{eqnarray}\label{Lagrangians}\nonumber
%&L_{\rm per}(x,\xi)= |\xi|^2+ V_{\rm per}(x),\quad
%L(x,\xi)= |\xi|^2+ V_{\rm per}(x) + W(x),\\
%&L_{\rm hom}(\xi)= f_{\rm hom}(\xi),
%\\
%\label{Hamiltonians}\nonumber
%&H_{\rm per}(x,\xi)= \frac14|\xi|^2- V_{\rm per}(x),\qquad
%H(x,\xi)= \frac14|\xi|^2- V_{\rm per}(x) - W(x),\\
%&H_{\rm hom}(\xi)= L^*_{\rm hom}(\xi)= f^*_{\rm hom}(\xi),
%\end{eqnarray}
%
\begin{eqnarray}\label{Lagrangians}
&L_{\rm per}(x,\xi)= |\xi|^2+ V_{\rm per}(x),&
L(x,\xi)= |\xi|^2+ V_{\rm per}(x) + W(x),
\\
\label{Hamiltonians} 
&H_{\rm per}(x,\xi)= \tfrac14|\xi|^2- V_{\rm per}(x),&
H(x,\xi)= \tfrac14|\xi|^2- V_{\rm per}(x) - W(x),
\\
&L_{\rm hom}(\xi)= f_{\rm hom}(\xi),
&H_{\rm hom}(\xi)%= L^*_{\rm hom}(\xi)
= f^*_{\rm hom}(\xi),
\end{eqnarray}
where $*$ denotes the Fenchel conjugate.

Analogous results can be obtained in the case of more general Lagrangians as in Section \ref{sec:ext} and the related Hamiltonians.

\subsection{Steady-state Hamilton--Jacobi equations}

We fix $\lambda>0$. We observe that, thanks to \eqref{WBUC},  for every $\e>0$ there exists a unique viscosity solution $U_\e\in W^{1,\infty}(\mathbb R^d)$ of the Hamilton-Jacobi equation 
\begin{equation}\label{H-prob-inf}
\lambda U_\e(x) + H\big(\frac{x}\e, \nabla U_\e(x)\big)=0,
\end{equation}
and likewise  there exists a unique viscosity solution $U\in W^{1,\infty}(\mathbb R^d)$ of the Hamilton-Jacobi equation 
\begin{equation}\label{H-hom-prob-inf}
\lambda U(x) + H_{\rm hom}(\nabla U(x))=0.
\end{equation}
The existence is proved as a particular case of \cite[Theorem 2.1]{Lions} and the uniqueness can be deduced from the example after \cite[Remark 1.15]{Lions}. 
In the unperturbed case, when $W=0$, equation \eqref{H-prob-inf} reduces to 
\begin{equation}\label{H-prob-per}
\lambda U_\e(x) + H_{\rm per}\big(\frac{x}\e, \nabla U_\e(x)\big)=0.
\end{equation}
The convergence of the solutions of \eqref{H-prob-per} to the solution of \eqref{H-hom-prob-inf} can be obtained using the techniques introduced in the fundamental unpublished paper by Lions, Papanicolaou, and Varadhan \cite{LPV} (see also \cite{MR1159184,MR1871349}).

We prove the following stability result.

\begin{theorem}[Stability for steady-state Hamilton--Jacobi equations]
Let $V_{\rm per} \colon\mathbb R^d\to\mathbb R$ be a continuous $1$-periodic function, and let 
$W \colon\mathbb R^d\to\mathbb R$ be a bounded and uniformly continuous function satisfying \eqref{2-2}. With fixed $\lambda >0$, for every $\e>0$ let $U_\e\in  W^{1,\infty}(\mathbb R^d)$ be the unique viscosity solution of \eqref{H-prob-inf}
and let $U\in  W^{1,\infty}(\mathbb R^d)$ be the unique viscosity solution of \eqref{H-hom-prob-inf}.
Then $U_\e$ tends to $U$ uniformly on compact sets of $\mathbb R^d$.
\end{theorem}

\begin{proof} By a classical result, the solutions $U_\e$ and $U$ are given by
\begin{eqnarray}
&\displaystyle U_\e(x)=\min\Big\{\int_0^{+\infty} L\Big(\frac{u(t)}\e,u'(t)\Big)e^{-\lambda t}dt: 
u\in H^1_\lambda((0,+\infty);\mathbb R^d),\ u(0)=x\Big\},
\\
&\displaystyle U(x)=\min\Big\{\int_0^{+\infty} L_{\rm hom}(u'(t))e^{-\lambda t}dt: 
u\in H^1_\lambda((0,+\infty);\mathbb R^d),\ u(0)=x\Big\}.
\end{eqnarray}
For a proof we refer to \cite[Chapter III, Proposition 2.8]{MR1484411} (see also \cite{MR0882926,MR2360607}).
Corollary \ref{cor-con-pu} gives the pointwise convergence of $U_\e(x)$ to $U(x)$.
In order to prove the uniform convergence on compact sets it suffices to show a uniform bound for the solutions in $W^{1,\infty}(\mathbb R^d)$. First, as $U_\e$ is concerned we note that, since $\frac1\lambda\inf( V_{\rm per}+W)$ and $\frac1\lambda\sup( V_{\rm per}+W)$ are a viscosity subsolution and a viscosity supersolution of \eqref{H-prob-inf}, respectively, by the comparison principle (see for instance \cite[Chapter II, Theorem 3.5]{MR1484411}) we have 
$$
\inf( V_{\rm per}+W)\le \lambda U_\e(x) \le \sup( V_{\rm per}+W)
$$
for every $x\in\mathbb R^d$. From this estimate we obtain a uniform bound for $U_\e$ and by coerciveness thus for $\nabla U_\e$. Indeed, from equation \eqref{H-prob-inf} we have 
$$
\frac14|\nabla U_\e(x)|^2 +\lambda U_\e(x)= V_{\rm per} \Big(\frac{x}\e\Big)+W \Big(\frac{x}\e\Big)
$$
and hence obtain a bound for $|\nabla U_\e(x)|$ uniform with respect to $\e$ and $x$.
The uniform convergence on compact sets follows from Ascoli--Arzel\`a's theorem.
\end{proof}

\subsection{Time-dependent Hamilton--Jacobi equations}
In this section $\Phi$ will be a fixed bounded uniformly continuous function, and $\nabla$ will denote the gradient with respect to $x$.
It is known that the Cauchy problem for the evolution equations on $\mathbb R^d\times [0,+\infty)$ given by
\begin{equation}\label{H-hom-probe}
\begin{cases}\partial_t U_\e(x,t) + H\big(\frac{x}\e, \nabla U_\e(x,t)\big)=0,\\
U_\e(x,0)= \Phi(x)
\end{cases}
\end{equation}
and 
\begin{equation}\label{H-hom-prob}
\begin{cases}\partial_t U(x,t) + H_{\rm hom}(\nabla U(x,t))=0,\\
U(x,0)= \Phi(x).
\end{cases}
\end{equation}
admit a unique viscosity solution  (see \cite[Theorem 9.1]{Lions} and \cite[Chapter 10]{Evans}). In  \cite{LPV} it is proved that when $W=0$ the solutions of 
\begin{equation}
\begin{cases}\partial_t U_\e(x,t) + H_{\rm per}\big(\frac{x}\e, \nabla U_\e(x,t)\big)=0,\\
U_\e(x,0)= \Phi(x)
\end{cases}
\end{equation}
converge uniformly to the viscosity solution of the homogenized equation \eqref{H-hom-prob}.
We now derive a stability result, showing that the same result holds also for the viscosity solutions corresponding to $H$.

\begin{theorem}[Stability for evolutionary Hamilton--Jacobi equations]
Let $V_{\rm per} \colon\mathbb R^d\to\mathbb R$ be a continuous $1$-periodic function, and let 
$W \colon\mathbb R^d\to\mathbb R$ be a bounded and uniformly continuous function satisfying \eqref{2-2}. Let $H$ be given by \eqref{Hamiltonians} and let $\Phi\colon\mathbb R^d\to \mathbb R $ be a bounded and uniformly continuous function. For every $\e>0$ let $U_\e$ be the viscosity solution of \eqref{H-hom-probe}
and let $U$ be the viscosity solution of \eqref{H-hom-prob}. Then $U_\e$ tends to $U$ uniformly on compact sets of $\mathbb R^d\times[0,+\infty)$.
\end{theorem}

\begin{proof}
To prove this result we use the characterization of viscosity solutions using the Lax formula (see for instance \cite{MR926208}): for every $x\in\mathbb R^d$ and $t>0$ we have
\begin{equation}\label{Lax-4}
U_\e(x,t)=\min\Bigl\{ \int_0^t L\Big(\frac{u(\tau)}\e, u'(\tau)\Big)d\tau +\Phi(u(0)): u\in H^1((0,t);\mathbb R^d), u(t)=x\Bigr\},
\end{equation}
and
\begin{eqnarray}\nonumber
U(x,t)&=&\min\Bigl\{ \int_0^t L_{\rm hom}(u'(\tau))d\tau +\Phi(u(0)): u\in H^1((0,t);\mathbb R^d), u(t)=x\Bigr\}\\ \label{Lax-3}
&=&\min\Bigl\{ t\, L_{\rm hom}\Big(\frac{x-y}t\Big) +\Phi(y): y\in\mathbb R^d\Bigr\}.
\end{eqnarray}

To prove the result it is enough to show that for all  $x_\e\to x_0$ we have 
\begin{eqnarray}\label{Lax-1}
\lim_{\e\to 0^+}U_\e(x_\e,t_\e)= U(x_0,t_0)\hbox{ if } t_\e\to t_0>0,\\ \label{Lax-2}
\lim_{\e\to 0^+}U_\e(x_\e,t_\e)= \Phi(x_0) \hbox{ if } t_\e\to 0,\hbox{ with } t_\e>0.
\end{eqnarray}
To prove \eqref{Lax-1}, we write \eqref{Lax-4} in the form 
\begin{equation}
U_\e(x,t)= \inf\big\{ S_\e(x,t,y)+ \Phi(y): y\in\mathbb R^d\big\},
\end{equation}
where
$$
S_\e(y, x,t)=\min\Bigl\{ \int_0^t L\Big(\frac{u(\tau)}\e, u'(\tau)\Big)d\tau : u\in H^1((0,t);\mathbb R^d), u(0)=y, u(t)=x\Bigr\}.
$$
We fix $x_\e\to x_0$  and $t_\e\to t_0>0$. We claim that for fixed $y$ we have 
$$\lim_{\e\to 0^+}S_\e(y,x_\e,t_\e)= S_{\hom} (y,x_0,t_0),
$$
where
\begin{eqnarray*}
S_{\rm hom}(y, x,t)&=&\min\Bigl\{ \int_0^t L_{\rm hom}(u'(\tau))d\tau : u\in H^1((0,t);\mathbb R^d), u(0)=y, u(t)=x\Bigr\}\\
&=& t\, L_{\rm hom}\Big(\frac{x-y}{t}\Big).
\end{eqnarray*}
Note that 
\begin{equation}
U(x,t)= \inf\big\{ S_{\rm hom}(x,t,y)+ \varphi(y): y\in\mathbb R^d\big\},
\end{equation}

We first prove some equi-continuity estimates for $S_\e$, uniform with respect to $\e$. 
With fixed $x,y$, we examine $S_\e(y,x,\cdot)$. If $t_1<t_2$ we have 
\begin{equation}\label{dista-2}
S_\e(y,x,t_2)\le S_\e(y,x, t_1)+ M(t_2-t_1),
\end{equation}
where $M=\max(V_{\rm per}+W)$. This is obtained by extending test functions to the constant value $x$ in $(t_1,t_2)$. Conversely, noting that 
$$
\int_0^{t_2} L\Big(\frac{u(\tau)}\e, u'(\tau)\Big)d\tau= \frac{t_2}{t_2}\int_0^{t_1} L\Big(\frac{v(\sigma)}\e, \frac{t_1}{t_2}v'(\sigma)\Big)d\sigma,
$$
where $u$ is a minimizer for $S_\e(y,x, t_2)$, and $v(\sigma)= u (\frac{t_2}{t_1}\sigma )$, and that
$$
\Big|L\Big(x, \frac{t_2}{t_1}\xi\Big)-L\Big(x, \xi\Big)\Big|= \Big(\Big(\frac{t_2}{t_1}\Big)^2-1\Big) |\xi|^2,
$$
we obtain
\begin{eqnarray}\label{dista-1}\nonumber
S_\e(y,x,t_1)&\le& S_\e(y,x, t_2)+ \Big(\Big(\frac{t_2}{t_1}\Big)^2-1\Big)\int_0^{t_2} |u'(\tau)|^2d\tau\\
&\le&S_\e(y,x, t_2)+ \Big(\Big(\frac{t_2}{t_1}\Big)^2-1\Big) \Big(Mt_2+\frac{|x-y|^2}{t_2}\Big).
\end{eqnarray}
 We finally deduce that, if $0<a\le t_1\le t_2\le b$ and $x,y\in B_R$ then
\begin{equation}\label{dista-5}
 |S_\e(y,x,t_1)- S_\e(y,x, t_2)|\le C(a,b,R) (t_2-t_1).
\end{equation}

As for the properties of $S_\e(y,\cdot,\cdot)$, for given $x_1,x_2\in B_R$, $0<a\le t_1\le t_2\le b$, and $\delta>0$, we first have 
\begin{equation}\label{dista-3}
S_\e(y,x_2,t_2+\delta)\le S_\e(y,x_1, t_1)+ M(t_2-t_1+\delta) +\frac{|x_2-x_1|^2}{t_2-t_1+\delta},
\end{equation}
obtained by extending test functions by an affine function in $(t_1,t_2)$. 
Using \eqref{dista-5}, we obtain 
\begin{equation}\label{dista-3}
S_\e(y,x_2,t_2)
\le S_\e(y,x_1, t_1)+ M(t_2-t_1+\delta) +\frac{|x_2-x_1|^2}{\delta}+ C(a,b,R) \delta,
\end{equation}

Conversely, we have, using \eqref{dista-5} in the first inequality and then
 \eqref{dista-3} with $t_1$ and $t_2$ replaced by $t_2$ and $t_2+(t_2-t_1)$, respectively,
 and $x_1$ and $x_2$ interchanged, 
\begin{eqnarray}\label{dista-4}\nonumber
S_\e(y,x_1,t_1)&\le& S_\e(y,x_1,t_2+(t_2-t_1))+2C(a,b,R) (t_2-t_1)\\
&\le&\nonumber
S_\e(y,x_2, t_2)+ M(t_2-t_1+\delta) +2R\frac{|x_2-x_1|}{\delta}+2C(a,b,R) (t_2-t_1).
\end{eqnarray}
Together with \eqref{dista-3}, this gives
\begin{eqnarray}\label{dista-6}\nonumber
|S_\e(y,x_1,t_1)-
S_\e(y,x_2, t_2)|\le M(t_2-t_1+\delta) +2R\frac{|x_2-x_1|}{\delta}+2C(a,b,R) (t_2-t_1+\delta).
\end{eqnarray}
If $|t_2-t_1|\le \delta$ and $|x_2-x_1|<\delta^2$ then  we have
$$
|S_\e(y,x_1,t_1)-
S_\e(y,x_2, t_2)|\le 2\big(M+R+2C(a,b,R)\big)\delta,
$$
which shows that $S_\e(y,\cdot,\cdot)$ are uniformly equicontinuous on compact subsets of $\mathbb R^d\times (0,+\infty)$. Finally, noting that a change of variables $\tau=t-\sigma$ interchanges symmetrically the role of $x$ and $y$ in the definition of $S_\e$, we infer that such functions are indeed 
uniformly equicontinuous on compact subsets of $\mathbb R^d\times\mathbb R^d\times (0,+\infty)$.

By the equicoerciveness and $\Gamma$-convergence with given boundary conditions, we have that $S_\e(y,x,t)$ converge to $S_{\rm hom}(y,x,t)$ for every $(y,x,t)$, and the uniformly equicontinuity just proven shows that this limit is uniform on compact subsets of $\mathbb R^d\times (0,+\infty)$. Recalling definition \eqref{Lax-3}, and noting that for given $x$ minimizers $y$ satisfy $|x-y|^2\le t^2M$ by Jensen's inequality and estimating $U_\e(x,t)$ by $S_\e(x,x,t)$, we  then deduce that $U_\e(x,t)$ converges uniformly to $U(x,t)$. 
%\bigskip
%\bigskip
%
%
%
%
%To prove the claim, we will use the $\Gamma$-convergence of these integrals on fixed intervals.
%To this end, we may write
%$$
%\int_0^{t_\e} L\Big(\frac{u(\tau)}\e, u'(\tau)\Big)d\tau= \frac{t_\e}{t_0}\int_0^{t_0} L\Big(\frac{v(\sigma)}\e, \frac{t_0}{t_\e}v'(\sigma)\Big)d\sigma,
%$$
%where $v(\sigma)= u (\frac{t_\e}{t_0}\sigma )$. Noting that 
%$$
%\Big|L\Big(x, \frac{t_0}{t_\e}\xi\Big)-L\Big(x, \xi\Big)\Big|= \Big|\Big|\frac{t_0}{t_\e}\Big|^2-1\Big||\xi|^2,
%$$
%from the $\Gamma$-convergence of $F_\e$ on fixed intervals we deduce that the functionals $\frac{t_\e}{t_0}\int_0^{t_0} L\Big(\frac{v(\sigma)}\e, \frac{t_0}{t_\e}v'(\sigma)\Big)d\sigma$ $\Gamma$-converge in the weak topology of $H^1((0,t_0);\mathbb R^d)$ to the functional $\int_0^{t_0} L_{\rm hom}(v'(\tau))d\tau$, preserving the boundary conditions $v(0)=y$ and $v(t_0)=x$. 
%
\end{proof}

\section{Negative perturbations}\label{nega}
In this section we give an example of a negative perturbation whose presence affects the form of the $\Gamma$-limit
in the spirit of \cite{MR4643677}, where a more general form of Hamiltonian (in particular not ``separable'', even in the perturbation) is considered.

We examine functionals 
$$
G_\e (u)=\int_0^1 \Big(|u'(t)|^2+ W\Big(\frac{u(t)}\e\Big)\Big)dt,
$$
defined in $H^1((0,1);\mathbb R^d)$,
where $W$ can be considered as a perturbation of the trivial periodic potential $V_{\rm per} =0$.

\begin{theorem} Let $ W\colon\mathbb R^d\to \mathbb R$ satisfy
$$
W(x)\le 0\hbox{ for all $x\in\mathbb R^d$} \quad\hbox{  and }\ 
\lim_{|x|\to+\infty} W(x)=0.
$$
Then 
$$
\Gamma\hbox{-}\lim_{\e\to 0} G_\e (u)=\int_0^1|u'(t)|^2\,dt +\inf W\,|\{t: u(t)=0\}|.
$$
\end{theorem}

Hence, if $W$ is not identically $0$ we have a different limit than in the unperturbed case, regardless of other conditions on $W$.
In particular, this holds for 
$W=-c\chi_{\{0\}}$, where $\chi$ denotes the characteristic function and $c>0$. Note that for such $W$ we have
$$
G_\e (u)=\int_0^1 |u'(t)|^2\,dt -c\,|\{t: u(t)=0\}|=\int_0^1 |u'(t)|^2\,dt +\inf W\,|\{t: u(t)=0\}|
$$
for all $\e>0$.

\begin{proof} We only consider the case $\inf W<0$.
With fixed $\delta>0$, let $x_\delta\in\mathbb R^d$ be such that $W(x_\delta)<\inf W+\delta< 0$, and set 
$$
W_\delta=W(x_\delta)\chi_{\{x_s\}}.
$$
We then have 
\begin{eqnarray*}
G_\e (u)&\le& G^\delta_\e (u):=\int_0^1 \Big(|u'(t)|^2+ W_\delta\Big(\frac{u(t)}\e\Big)\Big)dt\\
&=&\int_0^1 |u'(t)|^2dt -| W(x_\delta)||\{u=\e x_\delta\}|.
\end{eqnarray*}

Now if $u\in H^1((0,1),\mathbb R^d)$ we consider $u^\delta_\e= u+\e x_\delta$, which converges to $u$, and is such that
$$G^\delta_\e (u^\delta_\e)=\int_0^1 |u'(t)|^2dt -| W(x_\delta)||\{u=0\}|.$$ 
Hence, 
\begin{eqnarray*}
\limsup_{\e\to 0} G_\e (u)&\le& \int_0^1 |u'(t)|^2dt -| W(x_\delta)||\{u=0\}|\\
&\le&  \int_0^1 |u'(t)|^2dt +(\inf W+\delta)|\{u=0\}|,
\end{eqnarray*}
and, by the arbitrariness of $\delta$, the upper bound.

\smallskip
Conversely, with fixed $\delta>0$ we can consider $R_\delta>0$ such that
$$
W\ge-\delta+(\inf W)\chi_{\overline B_{R_\delta}},
$$
so that, for $\e$ small enough so that $\e R_\delta<\delta$, we have
\begin{eqnarray*}
G_\e (u)&\ge&\int_0^1|u'(t)|^2\,dt+\inf W|\{t: u(t)\in\overline B_{R_\delta}\}| -\delta\\
&\ge& \int_0^1|u'(t)|^2\,dt+\inf W\,|\{t: u(t)\in \overline B_{\delta}\}| -\delta.
\end{eqnarray*}
Noting that $u\mapsto -|\{t: u(t)\in  \overline B_{\delta}\}|$ is lower semicontinuous with respect to the convergence in $L^1$, if $u_\e\wto u$ weakly in $H^1(0,1)$ then we have 
$$
\liminf_{\e\to 0} G_\e (u) \ge \liminf_{\e\to 0} \int_0^1|u_\e'(t)|^2\,dt+ \liminf_{\e\to 0}\inf W\,|\{t: u_\e(t)\in\overline B_{\delta}\}|-\delta
$$
$$
\ge \int_0^1|u'(t)|^2\,dt+\inf W |\{t: u(t)\in \overline B_{\delta}\}|-\delta.
$$
Letting $\delta\to 0$ we obtain the claim.
\end{proof}

 \bibliographystyle{abbrv}
\bibliography{Bra-DM-LB-2023}

\end{document}